\documentclass[12pt,oneside, reqno]{amsart}
\usepackage{txfonts}
\usepackage{amsfonts}
\usepackage{mathrsfs}
\usepackage{amsmath}
\usepackage{amssymb, amsmath}
\usepackage{bbm}
\usepackage{stmaryrd}
\usepackage{color}
\usepackage{enumerate}
\newcommand{\be}{\begin{eqnarray}}
\newcommand{\ee}{\end{eqnarray}}
\newcommand{\ce}{\begin{eqnarray*}}
\newcommand{\de}{\end{eqnarray*}}
\newtheorem{theorem}{Theorem}[section]
\newtheorem{lemma}[theorem]{Lemma}
\newtheorem{remark}[theorem]{Remark}
\newtheorem{definition}[theorem]{Definition}
\newtheorem{proposition}[theorem]{Proposition}

\newtheorem{corollary}[theorem]{Corollary}

\def\[{{\Big[}}
\def\]{{\Big]}}
\def\<{{\langle}}
\def\>{{\rangle}}
\def\({{\Big(}}
\def\){{\Big)}}

\def\bx{{\mathbf{x}}}

\def\bt{\begin{theorem}}
\def\et{\end{theorem}}
\def\bl{\begin{lemma}}
\def\el{\end{lemma}}
\def\br{\begin{remark}}
\def\er{\end{remark}}
\def\bx{\begin{Example}}
\def\ex{\end{Example}}
\def\bd{\begin{definition}}
\def\ed{\end{definition}}
\def\bp{\begin{proposition}}
\def\ep{\end{proposition}}
\def\bc{\begin{corollary}}
\def\ec{\end{corollary}}

\def\geq{\geqslant}
\def\leq{\leqslant}

\def\vph{\varphi}

\def\bx{{\bf x}}
\allowdisplaybreaks

\allowdisplaybreaks
\numberwithin{equation}{section}
\begin{document}

\title{Central limit theorems for the derivatives of self-intersection local time for $d$-dimensional Brownian motion}

\date{}
\author{Xiaoyan Xu${}^{1}$, Xianye Yu${}^{2,\ast}$}
\dedicatory{${}^1$School of Mathematics, Sun Yat-sen University\\
 Guangzhou, Guangdong 510275, P. R. China\\
 ${}^2$School of Statistics and Mathematics, Zhejiang Gongshang University\\
 Hangzhou, Zhejiang 310018, P. R. China\\
Email: X. Xu: xiaoyanxu15@163.com, X. Yu: xianyeyu@gmail.com}
\thanks{${}^\ast$Corresponding author}
\thanks{{\it {\rm 2010} Mathematics Subject Classification:}
 60H07, 60G15, 60F05.}

\date{}

\begin{abstract}
Let $\{B_t,t\geq0\}$ be a d-dimensional Brownian motion. We prove that the approximation of the higher derivative of renormalized self-intersection local time
$$
\alpha_{d}^{(|k|)}(\epsilon)-E[\alpha_{d}^{(|k|)}(\epsilon)]:=\int_{0}^{1}\int_{0}^{s}\left(p^{(|k|)}_{d,\epsilon}(B_{s}-B_{r})-E[p^{(|k|)}_{d,\epsilon}(B_{s}-B_{r})]\right)drds,
$$
where the multiindex $k=(k_{1},\cdots,k_{d})$, $
p_{d,\epsilon}^{(|k|)}(x_1,x_2,\cdots,x_d):=\partial^{k_1}_{x_1}\partial^{k_2}_{x_2}\cdots\partial^{k_d}_{x_d}p_{d,\epsilon}(x_1,x_2,\cdots,x_d)
$ and $p_{d,\epsilon}(x)=\frac{1}{(2\pi\epsilon)^{d/2}}e^{-\frac{|x|^{2}}{2\epsilon}}, x\in\mathbb{R}^d$, satisfies the central limit theorems when renormalized by $(\log\frac{1}{\epsilon})^{-1}$ in the case $d=2$, $|k|=1$ and by $\epsilon^{\frac{d+|k|-3}{2}}$ in the case $d\geq 3$, $|k|\geq 1$, which gives a complete answer to the conjecture of Markowsky [In S\'{e}minaire de Probabiliti\'{e}s \uppercase\expandafter{\romannumeral10\romannumeral50\romannumeral4} (2012) 141-148 Springer]. We as well prove that the m-th chaotic component of $\alpha_{d}^{(|k|)}(\epsilon)-E[\alpha_{d}^{(|k|)}(\epsilon)]$ satisfies the central limit theorems when renormalized by a multiplicative factor in different cases.
\end{abstract}


\maketitle

\def\px{\pi(x)}
\def\bard{\bar{D}}
\def\lg{\langle}
\def\rg{\rangle}
\def\py{\pi(y)}
\def\nphi{\nabla\vph}
\def\w{\wedge}
\def\v{\vee}
\def\ve{\varepsilon}
\section{Introduction and main results}
Rogers and Walsh \cite{Rogers} introduced the integral functional
\begin{align*}
A(t,B_{t}):=\int_{0}^{t}1_{[0,\infty)}(B_{t}-B_{r})dr
\end{align*}
where $\{B_{t},t\geq0\}$ is a 1-dimensional Brownian motion. Denote by $L(t,x)$ the local time of Brownian motion, namely
$$
L(t,x)=\int^t_0\delta(B_s-x)ds,
$$
where $\delta$ is the Dirac function. Using $\frac{d}{dx}1_{[0,\infty)}=\delta(x)$ and $\frac{d^2}{dx^2}1_{[0,\infty)}=\delta'(x)$, Rosen \cite{Rosen} formally wrote the following It\^{o} formula without proof
\begin{align*}
A(t,B_{t})-\int_{0}^{t}L(s,B_{s})dB_{s}=t+\frac{1}{2}\int_{0}^{t}\int_{0}^{s}\delta'(B_{s}-B_{r})drds,
\end{align*}
and defined
\begin{align*}
 \alpha^{(1)}_{1}(t,x):=-\int_{0}^{t}\int_{0}^{s}\delta'(B_{s}-B_{r}-x)drds.
\end{align*}
Thus it is called the derivative of self-intersection local time of Brownian motion. Let $p_{1,\epsilon}$ be an approximate Dirac delta function, namely
\begin{align*}
p_{1,\epsilon}(x)=\frac{1}{\sqrt{2\pi\epsilon}}e^{-\frac{x^{2}}{2\epsilon}},\quad x\in\mathbb{R}.
\end{align*}
Denote
\begin{align*}
\alpha^{(1)}_{1}(t,x,\epsilon):=-\int_{0}^{t}\int_{0}^{s}p_{1,\epsilon}'(B_{s}-B_{r}-x)drds.
\end{align*}
Rosen \cite{Rosen} firstly derived the existence of derivative of self-intersection local time of $1$-dimensional Brownian motion, and showed that $\alpha^{(1)}_{1}(t,x,\epsilon)$ converges to $\alpha^{(1)}_{1}(t,x)$ in $L^p$ for all $p\in(0,\infty)$ when $\epsilon$ approaches to $0$. Markowsky \cite{Markowsky} proved the Tanaka-like formula
\begin{align*}
-\frac{1}{2}&\int_{0}^{t}\int_{0}^{s}\delta'(B_{s}-B_{r}-x)drds+\frac{1}{2}sgn(x)t\\
&=\int_{0}^{t}L(s,B_{s}-x)dB_{s}-\frac{1}{2}\int_{0}^{t}sgn(B_{t}-B_{r}-x)dr,
\end{align*}
where
\begin{equation*}
sgn(x)= \left\{
\begin{array}{ll}
-1,&\quad {\text { $x<0$}},\\
0,&\quad {\text { $x=0$}},\\
1,&\quad {\text { $x>0$}}.
\end{array}
\right.
\end{equation*}
Furthermore, Markowsky \cite{Markowsky-08} showed that
$$
\frac{1}{\log(1/\epsilon)}\int_{0}^{1}\int_{0}^{s}p'_{2,\epsilon}(B_{s}-B_{r})drds
$$
converges in law to a normal law as $\epsilon\rightarrow0$, where $\{B_{t},t\geq0\}$ is a 2-dimensional Brownian motion and for $x=(x_{1},x_{2})\in\mathbb{R}^{2}$, $$p_{2,\epsilon}(x)=\frac{1}{2\pi\epsilon}e^{-\frac{|x|^{2}}{2\epsilon}},\ \ \ p'_{2,\epsilon}(x_1,x_2)=\frac{\partial}{\partial x_1}p_{2,\epsilon}(x_1,x_2).$$
For the fractional Brownian motion case, the similar results could be found in \cite{Markowsky-2022,Hu-Nualart-05,Jaramillo-Nualart,Jung-Markowsky,Jung-Markowsky-15}. Then one can also consider the higher derivatives of self-intersection local time. For the multiindex $k=(k_{1},\cdots,k_{d})$ and $|k|:=k_1+k_2+\cdots+k_d$, we denote
\begin{align}\label{01-7-8}
\alpha^{(|k|)}_d(\epsilon):=\int_{0}^{1}\int_{0}^{s}p_{d,\epsilon}^{(|k|)}(B_{s}-B_{r})drds,
\end{align}
where $\{B_t,t\geq0\}$ is a d-dimensional Brownian motion and for $x=(x_1,x_2,\cdots,x_d)\in\mathbb{R}^d$,
$$p_{d,\epsilon}(x)=\frac{1}{(2\pi\epsilon)^{d/2}}e^{-\frac{|x|^{2}}{2\epsilon}},\ \ \
p_{d,\epsilon}^{(|k|)}(x_1,x_2,\cdots,x_d):=\partial^{k_1}_{x_1}\partial^{k_2}_{x_2}\cdots\partial^{k_d}_{x_d}p_{d,\epsilon}(x_1,x_2,\cdots,x_d).
$$
In~\cite{Markowsky2012}, Markowsky conjectured that for $d=1,$ $ |k|=2$ and some constant $\gamma>0$,
$$
\frac{1}{\log(1/\epsilon)^\gamma}\left(\alpha^{(2)}_1(\epsilon)-E[\alpha^{(2)}_d(\epsilon)]\right)
$$
converges in law to a normal law as $\epsilon$ approaches to zero.
For d-dimensional Brownian motion and $|k|\geq1$, Markowsky conjectured that for $d+|k|\geq4$,
\begin{equation*}
\epsilon^{\frac{d+|k|-3}{2}}\left(\alpha^{(|k|)}_d(\epsilon)-E[\alpha^{(|k|)}_d(\epsilon)]\right)
\end{equation*}
converges in law to a normal law as $\epsilon$ approaches to zero. The aim of this paper is to consider the precise asymptotic behaviors of $\alpha^{(|k|)}_d(\epsilon)-E[\alpha^{(|k|)}_d(\epsilon)]$ and to answer the above two conjectures completely.

For $|k|=0$, the self-intersection local time has been studied by many authors. When $d=1$, it is clear that $\alpha^{(0)}_1(\epsilon)$ converge in $L^2$
when $\epsilon$ approaches to zero. Unfortunately, $\alpha^{(0)}_d(\epsilon)$ does not converge in $L^2$ for any $d\geq2$ as $\epsilon$ tends to zero. In 1969, in order to construct the Edwards polymer model, Varadhan~\cite{Varadhan} adopted the renormalization and proved that $\alpha^{(0)}_2(\epsilon)-E[\alpha^{(0)}_2(\epsilon)]$ converges in $L^2$ as $\epsilon$ tends to zero. For $d=3$, Yor~\cite{Yor} proved that
$$
\left(\frac{1}{\log(1/\epsilon)}\right)^{\frac{1}{2}}\left(\alpha^{(0)}_3(\epsilon)-E[\alpha^{(0)}_3(\epsilon)]\right)\xrightarrow{Law} \mathcal{N}(0,\sigma^2), \quad\text{ as }\  \epsilon\rightarrow 0,
$$
where $\sigma^2$ is a positive constant. Furthermore, Calais and Yor~\cite{ Calais-Yor} have proved that for $d\geq4$
$$
\epsilon^{\frac{d-3}{2}}\left(\alpha^{(0)}_d(\epsilon)-E[\alpha^{(0)}_d(\epsilon)]\right)
$$
converges in law to a normal distribution as $\epsilon$ tends to zero. Along this line, we obtain the following results.

\begin{theorem}\label{th1}
Let $\alpha^{(|k|)}_d(\epsilon)$ be defined in~\eqref{01-7-8}. We have:
\begin{itemize}
\item [(i)] If $d=1$ and $|k|=2$, then
\begin{equation*}
\lim\limits_{\epsilon\to 0}\left(\log\frac{1}{\epsilon}\right)^{-\gamma}E\left[\left(\alpha^{(|k|)}_d(\epsilon)-E[\alpha^{(|k|)}_d(\epsilon)]\right)^2\right]= \left\{
\begin{array}{ll}
\infty,&\quad {\text { $0<\gamma\le1$}},\\
0,&\quad {\text { $\gamma>1$}}.
\end{array}
\right.
\end{equation*}
\item [(ii)] If $d=2$ and $|k|=1$, then
\begin{equation*}
\lim\limits_{\epsilon\to 0}\left(\log\frac{1}{\epsilon}\right)^{-2}E\left[\left(\alpha^{(|k|)}_d(\epsilon)-E[\alpha^{(|k|)}_d(\epsilon)]\right)^2\right]= \frac{1}{4\pi^2}.
\end{equation*}
\item [(iii)] If $d=1\text{ or }d=2$, and $d+|k|\geq4$, then
\begin{equation*}
\lim\limits_{\epsilon\to 0}\epsilon^{\gamma}E\left[\left(\alpha^{(|k|)}_d(\epsilon)-E[\alpha^{(|k|)}_d(\epsilon)]\right)^2\right]= \left\{
\begin{array}{ll}
\infty,&\quad {\text { $0<\gamma\le d+|k|-3$}},\\
0,&\quad {\text { $\gamma>d+|k|-3$}}.
\end{array}
\right.
\end{equation*}
\item [(iv)] If $d\geq3$ and $|k|\geq1$, then
\begin{equation*}
\lim\limits_{\epsilon\to 0}\epsilon^{d+|k|-3}E\left[\left(\alpha^{(|k|)}_d(\epsilon)-E[\alpha^{(|k|)}_d(\epsilon)]\right)^2\right]=\sum_{m>\frac{|k|}{2}}\sigma_{2m-|k|,d}^{2},
\end{equation*}
where
\begin{align}\label{sigma-1}
	\sigma_{2m-|k|,d}^{2}=\beta_{2m-|k|,d}\phi(m,d,|k|),
	\end{align}
	\begin{align*}
	\beta_{2m-|k|,d}:=\sum_{\substack{m_{1}+\cdots+m_{d}=m\\m_{j}\ge\frac{k_{j}}{2},\ j=1,\cdots,d}}\frac{((2m_{1})!\cdots(2m_{d})!)^{2}}{(2\pi)^{d}2^{2m}(2m_{1}-k_{1})!\cdots(2m_{d}-k_{d})!(m_{1}!)^{2}\cdots(m_{d}!)^{2}},
\end{align*}
and
	\begin{align*}\label{eq1.3-5-9}
	 \phi(m,d,|k|):=\frac{2\Gamma(2m-|k|+1)\Gamma(d+|k|-3)}{\Gamma(2m+d-2)(m+\frac{d}{2}-1)}\left(\frac{1}{m+\frac{d}{2}-1}+\frac{1}{m+\frac{d}{2}-2}\right).
	\end{align*}
\end{itemize}
\end{theorem}
\begin{remark}\label{re1-3-13}
If $|k|=k_1+k_2+\cdots+k_d$ is an odd number, it is trivial that
$$\alpha^{(|k|)}_d(\epsilon)-E[\alpha^{(|k|)}_d(\epsilon)]=\alpha^{(|k|)}_d(\epsilon).$$
For simplicity, we still denote $|k|:=k$ for $d=1$. When $d=2$ and $|k|=1$, we let $k_1=1$ and $k_2=0$ for multiindex $k=(k_1,k_2)$.
\end{remark}

By the technique of Wiener chaos expansion (see \cite{Nualart,Markowsky2012}), we can write
\begin{equation}\label{chaos-expansion}
\alpha^{(|k|)}_d(\epsilon)-E[\alpha^{(|k|)}_d(\epsilon)]=\sum_{m>\frac{|k|}{2}}I_{2m-|k|}(f_{2m-|k|,d,\epsilon}),
\end{equation}
where $f_{2m-|k|,d,\epsilon}$ is defined in Section~\ref{sec2}. We describe the limit behaviors of chaotic component $I_{2m-|k|}(f_{2m-|k|,d,\epsilon})$ as follows.

\begin{theorem}\label{th1.3-7-26}
Let the chaotic components $I_{2m-|k|}(f_{2m-|k|,d,\epsilon})$ be defined in~\eqref{chaos-expansion}. We have:
\begin{itemize}
\item [(i)] If $d=1$ and $|k|=2$, then for $m\geq2$
\begin{equation*}
\left(\log\frac{1}{\epsilon}\right)^{-\frac{1}{2}}I_{2m-2}(f_{2m-2,1,\epsilon})\xrightarrow{Law} \mathcal{N}\left(0,\frac{(2m-4)!}{\pi2^{2m-7}((m-2)!)^2}\right), \quad \text{ as }\  \epsilon\rightarrow 0.
\end{equation*}
\item [(ii)] If $d=2$ and $|k|=1$, then for $m=1$
\begin{equation*}
\left(\log\frac{1}{\epsilon}\right)^{-1}I_{1}(f_{1,2,\epsilon})\xrightarrow{Law} \mathcal{N}\left(0,\frac{1}{4\pi^2}\right), \quad \text{ as }\  \epsilon\rightarrow 0,
\end{equation*}
and for $m\geq2$
\begin{equation*}
\left(\log\frac{1}{\epsilon}\right)^{-\frac{1}{2}}I_{2m-1}(f_{2m-1,2,\epsilon})\xrightarrow{Law} \mathcal{N}\left(0,\frac{2\beta_{2m-1,2}(2m-1)}{m^2(m-1)}\right), \quad \text{ as }\  \epsilon\rightarrow 0,
\end{equation*}
where
\begin{align*}
\beta_{2m-1,2}:=\sum_{\substack{m_{1}+m_{2}=m\\m_1\geq 1,\  m_2\geq0}}\frac{((2m_{1})!(2m_{2})!)^{2}}{(2\pi)^{2}2^{2m}(2m_{1}-1)!(2m_{2})!(m_{1}!)^{2}(m_{2}!)^{2}}.
\end{align*}	

\item [(iii)] If $d\geq1$, $|k|\geq1$ and $d+|k|\geq4$, then
\begin{equation*}
\epsilon^{\frac{d+|k|-3}{2}}I_{2m-|k|}(f_{2m-|k|,d,\epsilon})\xrightarrow{Law} \mathcal{N}(0,\widehat{\sigma}_{2m-|k|,d}^{2}), \quad \text{ as }\  \epsilon\rightarrow 0,
\end{equation*}
where
\begin{align*}
\widehat{\sigma}_{2m-|k|,d}^{2}= \left\{
\begin{array}{ll}
\frac{(2m-1)!(k-3)!}{2^{2m-3}\pi((m-1)!)^{2}}\left(\frac{1}{m-\frac{1}{2}}+\frac{1}{m-\frac{3}{2}}\right),&\quad {\text { $d=1$}},\\
\sigma_{2m-|k|,d}^{2},&\quad {\text { $d\geq2$}},
\end{array}
\right.
\end{align*}
and $\sigma_{2m-|k|,d}^{2}$ is defined in~\eqref{sigma-1}.

\end{itemize}
\end{theorem}

The next results are the central limit theorems for $\alpha^{(|k|)}_d(\epsilon)-E[\alpha^{(|k|)}_d(\epsilon)]$ in two cases, and we give the complete answer to the conjecture of Markowsky, which is described in~\cite{Markowsky2012}.
\begin{theorem}\label{th-sum-convergence}
	Let $\alpha^{(|k|)}_d(\epsilon)$ be defined in~\eqref{01-7-8}. We have:
\begin{itemize}
\item [(i)] If $d=2$ and $|k|=1$, then
\begin{equation*}
\left(\log\frac{1}{\epsilon}\right)^{-1}\left(\alpha^{(|k|)}_d(\epsilon)-E[\alpha^{(|k|)}_d(\epsilon)]\right)\xrightarrow{Law} \mathcal{N}\left(0,\frac{1}{4\pi^2}\right), \quad \text{ as }\  \epsilon\rightarrow 0.
\end{equation*}
\item [(ii)] If $d\geq3$ and $|k|\geq1$, then
\begin{align*}
		\epsilon^{\frac{d+|k|-3}{2}}\left(\alpha_{d}^{(|k|)}(\epsilon)-E[\alpha_{d}^{(|k|)}(\epsilon)]\right)\xrightarrow{Law} \mathcal{N}\left(0,\sum_{m>\frac{|k|}{2}}\sigma_{2m-|k|,d}^{2}\right), \quad \text{ as }\  \epsilon\rightarrow 0,
	\end{align*}
where $\sigma_{2m-|k|,d}^{2}$ is defined in~\eqref{sigma-1}.
\end{itemize}
\end{theorem}

The paper is organized as follows. In section 2, we present some preliminaries and the Wiener chaos expansion of $\alpha_{d}^{(|k|)}(\epsilon)$. In section 3, we investigate the precise asymptotic behavior of the $2m-|k|$-th chaotic component of $\alpha_{d}^{(|k|)}(\epsilon)$ when $\epsilon$ approaches to $0$. In Section $4$, we borrow some ideas from Hu and Nurlart~\cite{Hu-Nualart-05} and present the central limit theorems of its chaotic components and then obtain the central limit theorems for $\alpha_{d}^{(|k|)}(\epsilon)-E[\alpha_{d}^{(|k|)}(\epsilon)]$ in two different cases. Finally, some technical lemmas are proved in Section 5.

\section{Preliminaries and the Wiener chaos expansion of $\alpha_{d}^{(|k|)}(\epsilon)$}\label{sec2}
\subsection{Preliminaries}
Let $\{B_t,t\geq0\}$ be 1-dimensional Brownian motion, which is defined on a probability space $(\Omega,\mathcal{F},\mathbb{P})$. We denote by $\mathfrak{H}$ the Hilbert space obtained by the completion of the space of step functions endowed with the inner product
\begin{align*}
\langle 1_{_{[a,b]}},1_{[c,d]}\rangle _{\mathfrak{H}}:=E[(B_{b}-B_{a})(B_{d}-B_{c})].
\end{align*}
The mapping $1_{[0,t]}\to B_{t}$ can be extended to a linear isometry between $\mathfrak{H}$ and a Gaussian subspace of $L^{2}(\Omega,\mathcal{F},\mathbb{P})$. We denote by $B(h)$ the image of $h\in\mathfrak{H}$ by this isometry. For any integer $m\in\mathbb{N}$, we denote by $\mathfrak{H}^{\otimes m}$ and $\mathfrak{H}^{\odot m}$ the $m$-th tensor product of $\mathfrak{H}^{\otimes m}$, and the $m$-th symmetric tensor product of $\mathfrak{H}^{\otimes m}$, respectively.
Let $\{e_{i},i\ge1\}$ be a complete orthonormal system in $\mathfrak{H}$. For every $p=0,\cdots,m$ and for every $f\in\mathfrak{H}^{\odot m}$, we define the contraction of $f$ of order $p$ to be the element of $\mathfrak{H}^{\otimes2(m-p)}$ defined by
\begin{align*}
f\otimes_{p}f=\sum_{i_{1},\cdots,i_{p}=1}^{\infty}\langle f, e_{i_{1}}\otimes\cdots\otimes e_{i_{p}}\rangle_{\mathfrak{H}^{\otimes p}}\otimes \langle f, e_{i_{1}}\otimes\cdots\otimes e_{i_{p}}\rangle_{\mathfrak{H}^{\otimes p}}.
\end{align*}

The $m$-th Wiener chaos of $L^{2}(\Omega,\mathcal{F},\mathbb{P})$, denoted by $\mathcal{H}_m$, is the closed subspace of $L^{2}(\Omega,\mathcal{F},\mathbb{P})$ generated by the variables $\{H_{m}(B(h)),\ h\in \mathfrak{H},\ \ \|h\|_{\mathfrak{H}}=1\}$, where $H_{m}$ is the $m$-th Hermite polynomial, defined by
\begin{align*}
H_{m}(x):=\frac{(-1)^{m}}{\sqrt{m!}}e^{\frac{x^{2}}{2}}\frac{d^{m}}{dx^{m}}(e^{-\frac{x^{2}}{2}}).
\end{align*}
The mapping $I_{m}(h^{\otimes m})=\sqrt{m!}H_{m}(B(h))$ provides a linear isometry between $\mathfrak{H}^{\odot m}$ (equipped with the norm $\sqrt{m!}\|\cdot\|_{\mathfrak{H}^{\otimes m}}$) and $\mathcal{H}_m$ (equipped with the $L^{2}$-norm).

\subsection{The Wiener chaos expansion of $\alpha_{d}^{(|k|)}(\epsilon)$}
Next we will compute the Wiener chaos expansion of the approximation of the self-intersection local time
\begin{align*}
		\alpha_{d}^{(|k|)}(\epsilon)=\int_{0}^{1}\int_{0}^{s}p_{d,\epsilon}^{(|k|)}(B_{s}-B_{r})drds,
\end{align*}
where the multiindex $k=(k_{1},\cdots,k_{d})$, $|k|=\sum_{i=1}^{d}k_{i}$ and $\{B_t,t\geq0\}$ is a $d$-dimensional Brownian motion.

Given a multiindex $\mathbf{i}_{n,d}=(i_{1},\dots,i_{n})$, $1\leq i_{l}\leq d$, for $l=1,\cdots,n$, we set
\begin{align*}
A(\mathbf{i}_{n,d})=E(\xi_{i_{1}}\cdots \xi_{i_{n}}),
\end{align*}
where $\{\xi_{i},\ 1\leq i\leq d\}$ are independent centered Gaussian random variables, and $n_{j}$ denotes the number of component of $\mathbf{i}_{n,d}$ equal to $j$.
Notice that if $n=2m$ is even,
\begin{align*}
A(\mathbf{i}_{2m,d})=\frac{(2m_{1})!\cdots (2m_{d})!}{m_{1}!\cdots m_{d}!2^{m}},
\end{align*}
otherwise $A(\mathbf{i}_{n,d})=0$.

\begin{lemma}\label{sec2-lem2.1}
	Let the integers $d\geq 1$, $|k|\geq1$ and $m>\frac{|k|}{2}$. We have
	\begin{align*}
	\alpha_{d}^{(|k|)}(\epsilon)=\sum_{m\geq\frac{|k|}{2}}I_{2m-|k|}(f_{2m-|k|,d,\epsilon}),
	\end{align*}
	where $m_{j}\geq\frac{k_{j}}{2}$ and
	\begin{align}\label{sec2-1}
	 f_{2m-|k|,d,\epsilon}(\mathbf{i}_{2m,d},r_{1},\cdots,r_{2m-|k|})=\frac{(-1)^{m}A(\mathbf{i}_{2m,d})}{(2\pi)^{\frac{d}{2}}(2m-|k|)!}\int_{0}^{1}\int_{0}^{t}\frac{\prod_{j=1}^{2m-|k|}1_{(s,t]}(r_{j})}{(t-s+\epsilon)^{m+\frac{d}{2}}}dsdt.
	\end{align}

In particular, we have for $d=1$ and $|k|\geq2$
	\begin{align}\label{sec2-2}
	 f_{2m-|k|,1,\epsilon}(r_{1},\cdots,r_{2m-|k|})=\frac{(-1)^{m}(2m)!}{\sqrt{2\pi}(2m-|k|)!m!2^{m}}\int_{0}^{1}\int_{0}^{t}\frac{\prod_{j=1}^{2m-|k|}1_{(s,t]}(r_{j})}{(t-s+\epsilon)^{m+\frac{1}{2}}}dsdt.
	\end{align}
\end{lemma}
\begin{proof}
	By Stroock's formula~(see~\cite{Nualart}), we have
	\begin{align*}
	\alpha_{d}^{(|k|)}(\epsilon)=\sum_{n=0}^{\infty}I_{n}(f_{n,d,\epsilon}),
	\end{align*}
	where
	\begin{align*}
	f_{n,d,\epsilon}=\frac{1}{n!}\int_{0}^{1}\int_{0}^{t}E[D_{r_{1},\cdots,r_{n}}^{i_{1},\cdots,i_{n}}p_{d,\epsilon}^{(|k|)}(B_{t}-B_{s})]dsdt.
	\end{align*}
	for $1\leq i_{l}\leq d$, $0\leq r_{l}\leq 1$, $l=1,\cdots,n$. We can write
	\begin{align*}
	 E[D_{r_{1},\cdots,r_{n}}^{i_{1},\cdots,i_{n}}p_{d,\epsilon}^{(|k|)}(B_{t}-B_{s})]=E[\partial_{i_{1}}\cdots\partial_{i_{n}}\partial_{1}^{k_{1}}\cdots\partial_{d}^{k_{d}}p_{d,\epsilon}(B_{t}-B_{s})]\prod_{j=1}^{n}1_{(s,t]}(r_{j}).
	\end{align*}
	It follows from the Fourier transform that
	\begin{align*}
	 E&[\partial_{i_{1}}\cdots\partial_{i_{n}}\partial_{1}^{k_{1}}\cdots\partial_{d}^{k_{d}}p_{d,\epsilon}(B_{t}-B_{s})]\\
&=\frac{i^{n+|k|}}{(2\pi)^{d}}\int_{\mathbb{R}^{d}}q_{i_{1}}\cdots q_{i_{n}}q_{1}^{k_{1}}\cdots q_{d}^{k_{d}}E[e^{i\langle q, B_{t}-B_{s}\rangle}]e^{-\frac{\epsilon |q|^{2}}{2}}dq\\
	&=\frac{i^{n+|k|}}{(2\pi)^{d}}\int_{\mathbb{R}^{d}}q_{i_{1}}\cdots q_{i_{n}}q_{1}^{k_{1}}\cdots q_{d}^{k_{d}}e^{-\frac{ |q|^{2}(t-s+\epsilon)}{2}}dq\\
	&=\frac{i^{n+|k|}}{(2\pi)^{\frac{d}{2}}(t-s+\epsilon)^{\frac{n+|k|+d}{2}}}\int_{\mathbb{R}^{d}}\frac{1}{(2\pi)^{\frac{d}{2}}}q_{i_{1}}\cdots q_{i_{n}}q_{1}^{k_{1}}\cdots q_{d}^{k_{d}}e^{-\frac{|q|^{2}}{2}}dq\\
	&=\frac{i^{n+|k|}}{(2\pi)^{\frac{d}{2}}(t-s+\epsilon)^{\frac{n+|k|+d}{2}}}E[\xi_{i_{1}}\cdots\xi_{i_{n}}\xi_{1}^{k_{1}}\cdots\xi_{d}^{k_{d}}],	
	\end{align*}
	where $\{\xi_{i},\ 1\leq i\leq d\}$ are independent $N(0,1)$ random variables. When $n+|k|=2m$, $n_{j}$ denotes the number of $(i_{1},\cdots,i_{n})$ equal to j. Let $n_{j}+k_{j}=2m_{j}$, where $m_{j}\geq\frac{k_{j}}{2}$, for $j=1,\cdots, d$ and  $\sum_{j=1}^{d}m_{j}=m$. Obviously, $\sum_{j=1}^{d}n_{j}=n$ and $\sum_{j=1}^{d}k_{j}=|k|$. Notice that
	\begin{align*}
	E[\xi_{i_{1}}\cdots\xi_{i_{n}}\xi_{1}^{k_{1}}\cdots\xi_{d}^{k_{d}}]
	=E[\xi_{1}^{n_{1}+k_{1}}\cdots\xi_{d}^{n_{d}+k_{d}}]=\frac{(2m_{1})!\cdots (2m_{d})!}{m_{1}!\cdots m_{d}!2^{m}}=A(\mathbf{i}_{2m,d}).
	\end{align*}
    Hence we obtain (\ref{sec2-1}).

Moreover, (\ref{sec2-2}) follows from $A(\mathbf{i}_{2m,1})=\frac{(2m)!}{m!2^{m}}$ for $d=1$. Thus, we complete the proof.
\end{proof}

\section{The precise asymptotic behavior of the variance of $I_{2m-|k|}(f_{2m-|k|,d,\epsilon})$}\label{sec3}
In this section, we discuss the precise asymptotic behavior of the variance of chaotic component $I_{2m-|k|}(f_{2m-|k|,d,\epsilon})$ for any $d\geq1$ and $|k|\geq1$, and we also investigate the precise asymptotic behavior of $\alpha^{(|k|)}_d(\epsilon)-E[\alpha^{(|k|)}_d(\epsilon)]$. First we introduce some notations. For $d\geq1$ and $|k|\geq1$, we denote
\begin{align*}
V^{(2m-|k|,d)}(\epsilon):=E[I_{2m-|k|}(f_{2m-|k|,d,\epsilon})^{2}],
\end{align*}
where $2m>|k|$.
It follows from Lemma \ref{sec2-lem2.1} that
\begin{equation}\label{V}
\begin{split}
V&^{(2m-|k|,d)}(\epsilon)\\
&=(2m-|k|)!\sum_{m_{1}+\cdots+m_{d}=m}\frac{(2m-|k|)!}{(2m_{1}-k_{1})!\cdots(2m_{d}-k_{d})!}\left(\frac{A(\mathbf{i}_{2m,d})}{(2\pi)^{\frac{d}{2}}(2m-|k|)!}\right)^{2}\\
&\qquad \times\int_{0}^{1}\int_{0}^{t_{1}}\int_{0}^{1}\int_{0}^{t_{2}}\frac{(\langle 1_{[s_{1},t_{1}]},1_{[s_{2},t_{2}]}\rangle_{\mathfrak{H}})^{2m-|k|}}{(t_{1}-s_{1}+\epsilon)^{m+\frac{d}{2}}(t_{2}-s_{2}+\epsilon)^{m+\frac{d}{2}}}ds_{1}dt_{1}ds_{2}dt_{2}\\
&=\beta_{2m-|k|,d}\int_{0}^{1}\int_{0}^{t_{1}}\int_{0}^{1}\int_{0}^{t_{2}}\frac{([s_{1},t_{1}]\cap [s_{2},t_{2}])^{2m-|k|}}{(t_{1}-s_{1}+\epsilon)^{m+\frac{d}{2}}(t_{2}-s_{2}+\epsilon)^{m+\frac{d}{2}}}ds_{1}dt_{1}ds_{2}dt_{2}\\
&=2\beta_{2m-|k|,d}\left[\int_{\mathcal{T}_{1}}\frac{(t_{1}-s_{2})^{2m-|k|}}{(t_{1}-s_{1}+\epsilon)^{m+\frac{d}{2}}(t_{2}-s_{2}+\epsilon)^{m+\frac{d}{2}}}ds_{1}dt_{1}ds_{2}dt_{2}\right.\\
&\qquad\qquad+\left.\int_{\mathcal{T}_{2}}\frac{(t_{2}-s_{2})^{2m-|k|}}{(t_{1}-s_{1}+\epsilon)^{m+\frac{d}{2}}(t_{2}-s_{2}+\epsilon)^{m+\frac{d}{2}}}ds_{1}dt_{1}ds_{2}dt_{2}\right]\\
&:=2\beta_{2m-|k|,d}\Big[V_{1}^{(2m-|k|,d)}(\epsilon)+V_{2}^{(2m-|k|,d)}(\epsilon)\Big],
\end{split}
\end{equation}
where
\begin{equation}\label{beta}
\begin{split}
\beta_{2m-|k|,d}&=\sum_{\substack{m_{1}+\cdots+m_{d}=m\\m_{j}\ge\frac{k_{j}}{2},\ j=1,\cdots,d}}\frac{((2m-|k|)!)^2}{(2m_{1}-k_{1})!\cdots(2m_{d}-k_{d})!}\left(\frac{A(\mathbf{i}_{2m,d})}{(2\pi)^{\frac{d}{2}}(2m-|k|)!}\right)^{2}\\
&=\sum_{\substack{m_{1}+\cdots+m_{d}=m\\m_{j}\ge\frac{k_{j}}{2},\ j=1,\cdots,d}}\frac{((2m_{1})!\cdots(2m_{d})!)^{2}}{(2\pi)^{d}2^{2m}(2m_{1}-k_{1})!\cdots(2m_{d}-k_{d})!(m_{1}!)^{2}\cdots(m_{d}!)^{2}},
\end{split}
\end{equation}
and $\mathcal{T}_{1}$, $\mathcal{T}_{2}$ is defined by
\begin{align*}
\mathcal{T}_{1}=\{(s_{1},t_{1},s_{2},t_{2}):0<s_{1}<s_{2}<t_{1}<t_{2}<1\},\\
\mathcal{T}_{2}=\{(s_{1},t_{1},s_{2},t_{2}):0<s_{1}<s_{2}<t_{2}<t_{1}<1\}.
\end{align*}

\subsection{The precise asymptotic behavior of $V^{(2m-|k|,d)}(\epsilon)$}\label{sec3.1}
The precise asymptotic behavior of the variance $V^{(2m-|k|,d)}(\epsilon)$ is described in the following propositions for different cases.
\begin{proposition}\label{pro3.1}
Let $d=1$ and $|k|=2$. we have for $m\geq2$
\begin{equation*}
\lim\limits_{\epsilon\to 0}\left(\log\frac{1}{\epsilon}\right)^{-1}V^{(2m-2,1)}(\epsilon)= 2\beta_{2m-2,1}\left[\frac{1}{(m-\frac{1}{2})^{2}}+\frac{1}{(m-\frac{1}{2})(m-\frac{3}{2})}\right],
\end{equation*}
where
\begin{align*}
	\beta_{2m-2,1}=\frac{(2m-1)(2m-1)!}{\pi 2^{2m-1}((m-1)!)^{2}}.
	\end{align*}
\end{proposition}
\begin{proof}For $d=1$ and $|k|=2$, we have from~\eqref{V}
\begin{equation*}
\begin{split}
V^{(2m-2,1)}(\epsilon)&=2\beta_{2m-2,1}\int_{\mathcal{T}_{1}}\frac{(t_1-s_2)^{2m-2}}{(t_1-s_1+\epsilon)^{m+\frac{1}{2}}(t_2-s_2+\epsilon)^{m+\frac{1}{2}}}ds_{1}
dt_{1}ds_{2}dt_{2}\\
&\quad+2\beta_{2m-2,1}\int_{\mathcal{T}_{2}}\frac{(t_2-s_2)^{2m-2}}{(t_1-s_1+\epsilon)^{m+\frac{1}{2}}(t_2-s_2+\epsilon)^{m+\frac{1}{2}}}ds_{1}dt_{1}ds_{2}dt_{2}\\
&=2\beta_{2m-2,1}\Big[V_{1}^{(2m-2,1)}(\epsilon)+V_{2}^{(2m-2,1)}(\epsilon)\Big]
,
\end{split}
\end{equation*}
where
\begin{align*}
	\beta_{2m-2,1}=\frac{(2m-1)(2m-1)!}{\pi 2^{2m-1}((m-1)!)^{2}}.
	\end{align*}
 Making the change of variables $(s_{1},s_{2},t_{1},t_{2})=(s_{1},a:=s_{2}-s_{1},b:=t_{1}-s_{2},c:=t_{2}-t_{1})$, one can write
	\begin{align*}
	 V_{1}^{(2m-2,1)}(\epsilon)&=\int_{\mathcal{T}_{1}}\frac{(t_{1}-s_{2})^{2m-2}}{(t_{1}-s_{1}+\epsilon)^{m+\frac{1}{2}}(t_{2}-s_{2}+\epsilon)^{m+\frac{1}{2}}}ds_{1}dt_{1}ds_{2}dt_{2}\\
	&=\int_{0<s_{1}+a+b+c<1}\frac{b^{2m-2}}{(a+b+\epsilon)^{m+\frac{1}{2}}(b+c+\epsilon)^{m+\frac{1}{2}}}ds_{1}dadbdc\\
    &=\frac{1}{m-\frac{1}{2}}\int_{0<s_{1}+a+b<1}\frac{b^{2m-2}}{(a+b+\epsilon)^{m+\frac{1}{2}}(b+\epsilon)^{m-\frac{1}{2}}}ds_{1}dadb\\
    &\quad-\frac{1}{m-\frac{1}{2}}\int_{0<s_{1}+a+b<1}\frac{b^{2m-2}}{(a+b+\epsilon)^{m+\frac{1}{2}}(1-s_1-a+\epsilon)^{m-\frac{1}{2}}}ds_{1}dadb.\\
    &=\frac{1}{(m-\frac{1}{2})^{2}}\int_{0<s_{1}+b<1}\left[\frac{b^{2m-2}}{(b+\epsilon)^{2m-1}}-\frac{b^{2m-2}}{(1-s_1+\epsilon)^{m-\frac12}(b+\epsilon)^{m-\frac12}}\right]ds_{1}db\\
    &\quad-\frac{1}{(m-\frac{1}{2})(m-\frac{3}{2})}\int_{0<a+b<1}\frac{b^{2m-2}}{(a+b+\epsilon)^{m+\frac{1}{2}}(b+\epsilon)^{m-\frac{3}{2}}}dadb\\
     &\quad+\frac{1}{(m-\frac{1}{2})(m-\frac{3}{2})}\int_{0<a+b<1}\frac{b^{2m-2}}{(a+b+\epsilon)^{m+\frac{1}{2}}(1-a+\epsilon)^{m-\frac{3}{2}}}dadb.
	\end{align*}
By changing the coordinates $(a,b)=(\epsilon x,\epsilon y)$, we rewrite that
\begin{align*}
V_{1}&^{(2m-2,1)}(\epsilon)\\
&=\frac{1}{(m-\frac{1}{2})^{2}}\int_{0}^{\frac{1}{\epsilon}}\frac{y^{2m-2}}{(y+1)^{2m-1}}dy-\frac{\epsilon}{(m-\frac{1}{2})^{2}}\int_{0}^{\frac{1}{\epsilon}}\frac{y^{2m-1}}{(y+1)^{2m-1}}dy\\
&
-\frac{2\epsilon}{(m-\frac{1}{2})^{2}(m-\frac{3}{2})}\int_{0}^{\frac{1}{\epsilon}}\frac{y^{2m-2}}{(y+1)^{2m-2}}dy+
\frac{\epsilon^{m-\frac{1}{2}}(1+\epsilon)^{-m+\frac{3}{2}}}{(m-\frac{1}{2})^{2}(m-\frac{3}{2})}\int_{0}^{\frac{1}{\epsilon}}\frac{y^{2m-2}}{(y+1)^{m-\frac{1}{2}}}dy\\
&+\frac{\epsilon^{m+\frac{1}{2}}(1+\epsilon)^{-m+\frac{1}{2}}}{(m-\frac{1}{2})^{2}(m-\frac{3}{2})}\int_{0}^{\frac{1}{\epsilon}}\frac{y^{2m-2}}{(y+1)^{m-\frac{3}{2}}}dy\\
&+\frac{\epsilon}{(m-\frac{1}{2})(m-\frac{3}{2})}\int_{0<x+y<\frac{1}{\epsilon}}\frac{y^{2m-2}}{(x+y+1)^{m+\frac{1}{2}}(1+\frac{1}{\epsilon}-x)^{m-\frac{3}{2}}}dxdy\\
:&=\sum_{i=1}^{6}K_{i}^{(2m-2,1)}(\epsilon).
\end{align*} 	
By L'H\^{o}pital's rule, thus
\begin{align*}
\lim_{\epsilon\rightarrow0}\frac{K_{1}^{(2m-2,1)}(\epsilon)}{\log(1/\epsilon)}=\frac{1}{(m-\frac{1}{2})^{2}}.
\end{align*}	
It is trivial that for $i=2,\dots,5$,
\begin{align*}
\lim_{\epsilon\rightarrow0}\frac{K_{i}^{(2m-2,1)}(\epsilon)}{\log(1/\epsilon)}=0.
\end{align*}
Lemma \ref{lastsec-lem2} gives
\begin{align*}
	\lim_{\epsilon\rightarrow0}\frac{K_{6}^{(2m-2,1)}(\epsilon)}{\log(1/\epsilon)}=0.
\end{align*}
Summing up the above estimates,
\begin{equation}\label{sec3-3.4}
\lim_{\epsilon\rightarrow0}\frac{V_{1}^{(2m-2,1)}(\epsilon)}{\log(1/\epsilon)}=\frac{1}{(m-\frac{1}{2})^2}.
\end{equation}
Next we just need to deal with the term $V_{2}^{(2m-2,1)}(\epsilon)$. Similarly, changing the variables $(s_{1},s_{2},t_{2},t_{1})=(s_{1},a:=s_{2}-s_{1},b:=t_{2}-s_{2},c:=t_{1}-t_{2})$, we have
\begin{align*}
V_{2}^{(2m-2,1)}(\epsilon)=&\int_{\mathcal{T}_{2}}\frac{(t_{2}-s_{2})^{2m-2}}{(t_{1}-s_{1}+\epsilon)^{m+\frac{1}{2}}(t_{2}-s_{2}+\epsilon)^{m+\frac{1}{2}}}ds_{1}dt_{1}ds_{2}dt_{2}\\
=&\int_{0<s_{1}+a+b+c<1}\frac{b^{2m-2}}{(a+b+c+\epsilon)^{m+\frac{1}{2}}(b+\epsilon)^{m+\frac{1}{2}}}ds_{1}dadbdc\\
=&\frac{1}{m-\frac{1}{2}}\int_{0<s_{1}+a+b<1}\frac{b^{2m-2}}{(a+b+\epsilon)^{m-\frac{1}{2}}(b+\epsilon)^{m+\frac{1}{2}}}ds_{1}dadbdc\\
&+\frac{1}{m-\frac{1}{2}}\int_{0<s_{1}+a+b<1}\frac{b^{2m-2}}{(1-s_1+\epsilon)^{m-\frac{1}{2}}(b+\epsilon)^{m+\frac{1}{2}}}ds_{1}dadbdc\\
=&\frac{1}{(m-\frac{1}{2})(m-\frac{3}{2})}\left[\int_{0}^{\frac{1}{\epsilon}}\frac{y^{2m-2}}{(y+1)^{2m-1}}dy-\epsilon\int_{0}^{\frac{1}{\epsilon}}\frac{y^{2m-1}}{(y+1)^{2m-1}}dy\right.\\
&-\frac{\epsilon}{m-\frac{5}{2}}\int_{0}^{1}\frac{y^{2m-2}}{(y+1)^{2m-2}}dy+\frac{\epsilon^{m-\frac{3}{2}}(1+\epsilon)^{-m+\frac{5}{2}}}{m-\frac{5}{2}}\int_{0}^{\frac{1}{\epsilon}}\frac{y^{2m-2}}{(y+1)^{m+\frac{1}{2}}}dy\\
&-\frac{\epsilon}{m-\frac{5}{2}}\int_{0}^{\frac{1}{\epsilon}}\frac{y^{2m-2}}{(y+1)^{2m-2}}dy+\frac{\epsilon^{m-\frac{3}{2}}(1+\epsilon)^{-m+\frac{5}{2}}}{m-\frac{5}{2}}\int_{0}^{\frac{1}{\epsilon}}\frac{y^{2m-2}}{(y+1)^{m+\frac{1}{2}}}dy\\
&\left.+\frac{\epsilon^{m-\frac{3}{2}}}{(1+\epsilon)^{m-\frac{3}{2}}}\int_{0}^{\frac{1}{\epsilon}}\frac{y^{2m-2}}{(y+1)^{m+\frac{1}{2}}}dy-\frac{\epsilon^{m-\frac{1}{2}}}{(1+\epsilon)^{m-\frac{3}{2}}}\int_{0}^{\frac{1}{\epsilon}}\frac{y^{2m-1}}{(y+1)^{m+\frac{1}{2}}}dy\right]\\
&:=\frac{1}{(m-\frac{1}{2})(m-\frac{3}{2})}\sum_{i=1}^{8}G_{i}^{(2m-2,1)}(\epsilon),
\end{align*}
where we replace $b$ by $\epsilon y$. By L'H\^{o}pital's rule, it is obvious that
\begin{align*}
\lim_{\epsilon\rightarrow0}\frac{G_{1}^{(2m-2,1)}(\epsilon)}{\log(1/\epsilon)}=1,
\end{align*}
and for $i=2,\dots, 8$,
\begin{equation*}
\lim_{\epsilon\rightarrow0}\frac{ G_{i}^{(2m-2,1)}(\epsilon)}{\log(1/\epsilon)}=0.
\end{equation*}
Thus,
\begin{equation}\label{sec3-3.5}
\lim_{\epsilon\rightarrow0}\frac{V_{2}^{(2m-2,1)}(\epsilon)}{\log(1/\epsilon)}=\frac{1}{(m-\frac{1}{2})(m-\frac{3}{2})}.
\end{equation}
It follows from \eqref{sec3-3.4} and \eqref{sec3-3.5} that
\begin{align*}
&\lim\limits_{\epsilon\to 0}\left(\log\frac{1}{\epsilon}\right)^{-1}V^{(2m-2,1)}(\epsilon)\\
&=
\lim_{\epsilon\rightarrow0}\frac{2\beta_{2m-2,1}}{\log(1/\epsilon)}V_{1}^{(2m-2,1)}(\epsilon)+\lim_{\epsilon\rightarrow0}\frac{2\beta_{2m-2,1}}{\log(1/\epsilon)}V_{2}^{(2m-2,1)}(\epsilon)\\
&=\frac{2\beta_{2m-2,1}}{(m-\frac{1}{2})^2}+\frac{2\beta_{2m-2,1}}{(m-\frac{1}{2})(m-\frac{3}{2})}\\
&=2\beta_{2m-2,1}\left[\frac{1}{(m-\frac{1}{2})^{2}}+\frac{1}{(m-\frac{1}{2})(m-\frac{3}{2})}\right].
\end{align*}
\end{proof}
\begin{proposition}\label{pro3.2}
Let $d=2$ and $|k|=1$. We have for $m=1$

\begin{align*}
		\lim_{\epsilon\rightarrow 0}\left(\log\frac{1}{\epsilon}\right)^{-2}V^{(1,2)}(\epsilon)=\frac{1}{4\pi^{2}},
	\end{align*}
    and for $m\geq2$
    \begin{align*}
    	\lim_{\epsilon\rightarrow 0}\left(\log\frac{1}{\epsilon}\right)^{-1}V^{(2m-1,2)}(\epsilon)=\frac{2\beta_{2m-1,2}(2m-1)}{m^2(m-1)},
    \end{align*}
	where
    \begin{align}\label{014}
    	\beta_{2m-1,2}:=\sum_{\substack{m_{1}+m_{2}=m\\m_1\geq 1,\  m_2\geq 0}}\frac{((2m_{1})!(2m_{2})!)^{2}}{(2\pi)^{2}2^{2m}(2m_{1}-1)!(2m_{2})!(m_{1}!)^{2}(m_{2}!)^{2}}.
    \end{align}	
\end{proposition}
\begin{proof} For $d=2$ and $|k|=1$, it follows immediately from (\ref{V}) that
\begin{align*}
V^{(2m-1,2)}(\epsilon)&=2\beta_{2m-1,2}\int_{\mathcal{T}_1}\frac{(t_1-s_2)^{2m-1}}{(t_{1}-s_{1}+\epsilon)^{m+1}(t_{2}-s_{2}+\epsilon)^{m+1}}ds_{1}dt_{1}ds_{2}dt_{2}\\
&\quad +2\beta_{2m-1,2}\int_{\mathcal{T}_2}\frac{(t_2-s_2)^{2m-1}}{(t_{1}-s_{1}+\epsilon)^{m+1}(t_{2}-s_{2}+\epsilon)^{m+1}}ds_{1}dt_{1}ds_{2}dt_{2}\\
&=:2\beta_{2m-1,2}\Big[V_{1}^{(2m-1,2)}(\epsilon)+V_{2}^{(2m-1,2)}(\epsilon)\Big],
\end{align*}
where the multiindex $k=(1,0)$ (see Remark~\ref{re1-3-13}) and
    \begin{align*}
    	\beta_{2m-1,2}:=\sum_{\substack{m_{1}+m_{2}=m\\m_1\geq 1,\  m_2\geq 0}}\frac{((2m_{1})!(2m_{2})!)^{2}}{(2\pi)^{2}2^{2m}(2m_{1}-1)!(2m_{2})!(m_{1}!)^{2}(m_{2}!)^{2}}.
    \end{align*}
    For $m=1$, making the change of variables $(s_{1},s_{2},t_{1},t_{2})=(s_{1},a:=s_{2}-s_{1},b:=t_{1}-s_{2},c:=t_{2}-t_{1})$, we have
    \begin{align*}
    	 V_{1}^{(1,2)}(\epsilon)&=\int_{\mathcal{T}_{1}}\frac{(t_1-s_2)}{(t_{1}-s_{1}+\epsilon)^{2}(t_{2}-s_{2}+\epsilon)^{2}}ds_{1}dt_{1}ds_{2}dt_{2}\\
    	&=\int_{0<s_{1}+a+b+c<1}\frac{b}{(a+b+\epsilon)^{2}(b+c+\epsilon)^{2}}ds_{1}dadbdc\\
    	&=\int_{0<s_{1}+b<1}\Big[\frac{b}{(b+\epsilon)^{2}}-\frac{b}{(b+\epsilon)(1-s_{1}+\epsilon)}\Big]ds_{1}db\\
    	&\quad -\int_{0<a+b<1}\frac{b\log(1-a+\epsilon)}{(a+b+\epsilon)^{2}}dadb +\int_{0<a+b<1}\frac{b\log(b+\epsilon)}{(a+b+\epsilon)^{2}}dadb\\
    	 &=\int_{0}^{1}\frac{b}{(b+\epsilon)^{2}}db-\int_{0}^{1}\frac{b^{2}}{(b+\epsilon)^{2}}db-\int_{0}^{1}\frac{b\log(1+\epsilon)}{b+\epsilon}db+\int_{0}^{1}\frac{b\log(b+\epsilon)}{b+\epsilon}db\\
    	&\quad -\int_{0}^{1}\log(1-a+\epsilon)\log(1+\epsilon)da-\int_{0}^{1}\frac{\log(1-a+\epsilon)(a+\epsilon)}{(1+\epsilon)}da\\
    	&\quad +\int_{0}^{1}\log(1-a+\epsilon)\log(a+\epsilon)da+\int_{0}^{1}\log(1-a+\epsilon)da\\
    	&\quad +\int_{0}^{1}\frac{b\log(b+\epsilon)}{b+\epsilon}db-\int_{0}^{1}\frac{b\log(b+\epsilon)}{1+\epsilon}db\\
    	&=\sum_{i=1}^{10}K_{i}^{(1,2)}(\epsilon).
    \end{align*}
    By L'H\^{o}pital's rule, thus
    \begin{align*}    	
    	\lim_{\epsilon\rightarrow0}\frac{K_{1}^{(1,2)}(\epsilon)}{\log (1/\epsilon)}&=\lim_{\epsilon\rightarrow0}\frac{\int_{0}^{1}\frac{b}{(b+\epsilon)^{2}}db}{\log(1/\epsilon)}
    =\lim_{\epsilon\rightarrow0}\frac{\int_{0}^{1/\epsilon}\frac{b}{(b+1)^{2}}db}{\log(1/\epsilon)}=\lim_{N\rightarrow\infty}\frac{N^2}{(N+1)^2}=1,
    \end{align*}
    and
    \begin{align*}
    	\lim_{\epsilon\rightarrow0} \frac{K_{i}^{(1,2)}(\epsilon)}{\log(1/\epsilon)}=0,\ \  \text{ }\ \ i=2,\dots,10.
    \end{align*}
   Hence,
    \begin{align}\label{sec3-eq3.9}
    	\lim_{\epsilon\rightarrow0}\frac{V_{1}^{(1,2)}(\epsilon)}{(\log(1/\epsilon))^{2}}=0.
    \end{align}
    Similarly, changing the variables $(s_{1},s_{2},t_{2},t_{1})=(s_{1},a:=s_{2}-s_{1},b:=t_{2}-s_{2},c:=t_{1}-t_{2})$, we have for the term $V_{2}^{(1,2)}(\epsilon)$,
   \begin{align*}
   	 V_{2}^{(1,2)}(\epsilon)=&\int_{\mathcal{T}_{2}}\frac{(t_2-s_2)}{(t_{1}-s_{1}+\epsilon)^{2}(t_{2}-s_{2}+\epsilon)^{2}}ds_{1}dt_{1}ds_{2}dt_{2}\\
   	=&\int_{0<s_{1}+a+b+c<1}\frac{b}{(a+b+c+\epsilon)^{2}(b+\epsilon)^{2}}ds_{1}dadbdc\\
   	=&\int_{0<s_{1}+b<1}\frac{b\log(1-s_{1}+\epsilon)}{(b+\epsilon)^{2}}ds_{1}db-\int_{0<s_{1}+b<1}\frac{b\log(b+\epsilon)}{(b+\epsilon)^{2}}ds_{1}db\\
   	&\ \ -\int_{0<a+b<1}\frac{b\log(1+\epsilon)}{(b+\epsilon)^{2}}ds_{1}db-\int_{0<s_{1}+b<1}\frac{b\log(a+b+\epsilon)}{(b+\epsilon)^{2}}dadb\\
   	 =&2\int_{0}^{1}\frac{b(1+\epsilon)\log(1+\epsilon)}{(b+\epsilon)^{2}}db-\int_{0}^{1}\frac{b(1+\epsilon)}{(b+\epsilon)^{2}}db-2\int_{0}^{1}\frac{b\log(b+\epsilon)}{b+\epsilon}db\\
   	 &\ \ +\int_{0}^{1}\frac{b}{b+\epsilon}db  	 -\int_{0}^{1}\frac{b\log(b+\epsilon)}{(b+\epsilon)^{2}}db+\int_{0}^{1}\frac{b^{2}\log(b+\epsilon)}{(b+\epsilon)^{2}}db\\
   	 &\ \ -\int_{0}^{1}\frac{b\log(1+\epsilon)}{(b+\epsilon)^{2}}db-\int_{0}^{1}\frac{b^{2}\log(1+\epsilon)}{(b+\epsilon)^{2}}db\\
   	:=&\sum_{i=1}^{8}G_{i}^{(1,2)}(\epsilon).
   \end{align*}
   It is obvious that
  \begin{align*}
  \lim_{\epsilon\rightarrow0} \frac{G_{i}^{(1,2)}(\epsilon)}{(\log(1/\epsilon))^{2}}=0,\ \  \text{ }\ \ i=1,2,3,4,6,7,8,
   \end{align*}
and
   \begin{align*}
    \lim_{\epsilon\rightarrow0} \frac {G_{5}^{(1,2)}(\epsilon)}{(\log(1/\epsilon))^2}&= \lim_{\epsilon\rightarrow0}\frac{-\int_{0}^{1}\frac{b\log(b+\epsilon)}{(b+\epsilon)^{2}}db}{(\log(1/\epsilon))^2}\\
    &=\lim_{\epsilon\rightarrow0}\frac{-\log \epsilon \int^{\frac{1}{\epsilon}}_0\frac{b}{(b+1)^2}db}{(\log(1/\epsilon))^2}-\lim_{\epsilon\rightarrow0}
    \frac{\int^{\frac{1}{\epsilon}}_0\frac{b\log(b+1)}{(b+1)^2}db}{(\log(1/\epsilon))^2}\\
    &=1-\frac{1}{2}=\frac{1}{2}.
   \end{align*}
We then have
   \begin{align}\label{sec3-eq3.10}
   		\lim_{\epsilon\rightarrow0}\frac{1}{(\log(1/\epsilon))^{2}}V_{2}^{(1,2)}(\epsilon)=\frac{1}{2}.
   \end{align}
   It follows from~\eqref{sec3-eq3.9} and \eqref{sec3-eq3.10} that
   $$
   \lim_{\epsilon\rightarrow 0}\left(\log\frac{1}{\epsilon}\right)^{-2}V^{(1,2)}(\epsilon)=2\beta_{1,2}\times\frac{1}{2}=\frac{1}{4\pi^{2}}.
   $$

   Similarly, one can write for $m=2$,
   \begin{align*}
   	 V_{1}^{(3,2)}(\epsilon)&=\int_{0}^{1}\frac{b^{3}}{2^{2}(b+\epsilon)^{4}}db-\int_{0}^{1}\frac{b^{4}}{2^{2}(b+\epsilon)^{4}}db-\int_{0}^{1}\frac{b^{3}}{2^{2}(b+\epsilon)^{3}}db\\
    &\quad+\int_{0}^{1}\frac{b^{3}}{2^{2}(b+\epsilon)^{2}(1+\epsilon)}db
    -\int_{0}^{1}\frac{b^{3}}{2^{2}(b+\epsilon)^{3}}db+\int_{0}^{1}\frac{b^{3}}{2^{2}(b+\epsilon)(1+\epsilon)^{2}}\\
     &\quad+\int_{0<a+b<1}\frac{b^{3}}{2(a+b+\epsilon)^{3}(1-a+\epsilon)}dadb,
   \end{align*}
   and
   \begin{align*}
   	 V_{2}^{(3,2)}(\epsilon)&=\int_{0}^{1}\frac{b^{3}}{2(b+\epsilon)^{4}}db-\int_{0}^{1}\frac{b^{4}}{2(b+\epsilon)^{4}}db-\int_{0}^{1}\frac{b^{3}\log(1+\epsilon)}{2(b+\epsilon)^{3}}db\\
     &\quad+\int_{0}^{1}\frac{b^{3}}{2(b+\epsilon)^{3}}db
   	 -\int_{0}^{1}\frac{b^{3}\log(1+\epsilon)}{2(b+\epsilon)^{3}}db+\int_{0}^{1}\frac{b^{3}\log(b+\epsilon)}{2(b+\epsilon)^{3}}db\\
    &\quad+\int_{0}^{1}\frac{b^{3}}{2(b+\epsilon)^{3}(1+\epsilon)}db
   	-\int_{0}^{1}\frac{b^{4}}{2(b+\epsilon)^{3}(1+\epsilon)}db.\\
   \end{align*}
Hence,
\begin{equation}\label{sec3-eq3.11}
\begin{split}
 &\lim_{\epsilon\rightarrow 0}\frac{1}{\log(1/\epsilon)}\left(V^{(3,2)}_1(\epsilon)+V^{(3,2)}_2(\epsilon)\right)\\
 =&\lim_{\epsilon\rightarrow 0}\frac{1}{\log(1/\epsilon)}\left(\int_{0}^{1}\frac{b^{3}}{2^{2}(b+\epsilon)^{4}}db+\int_{0}^{1}\frac{b^{3}}{2(b+\epsilon)^{4}}db\right)\\
 =&\lim_{\epsilon\rightarrow 0}\frac{1}{\log(1/\epsilon)}\left(\int_{0}^{1/\epsilon}\frac{b^{3}}{2^{2}(b+1)^{4}}db+\int_{0}^{1/\epsilon}\frac{b^{3}}{2(b+1)^{4}}db\right)\\
 =&\frac{1}{2^2}+\frac{1}{2}=\frac{3}{4}.
 \end{split}
\end{equation}

For $m\geq3$, with the same arguments, we have
   \begin{align*}
   	 V_{1}^{(2m-1,2)}(\epsilon)=&\frac{1}{m^{2}}\int_{0}^{1}\frac{b^{2m-1}}{(b+\epsilon)^{2m}}db-\frac{1}{m^{2}}\int_{0}^{1}\frac{b^{2m}}{(b+\epsilon)^{2m}}db\\
     &-\frac{2}{m^{2}(m-1)}\int_{0}^{1}\frac{b^{2m-1}}{(b+\epsilon)^{2m-1}}db\\
   	&
   	+
   	 \frac{1}{m^{2}(m-1)(1+\epsilon)^{m}}\int_{0}^{1}\frac{b^{2m-1}}{(b+\epsilon)^{m}}db\\
     &+\frac{1}{m^{2}(m-1)(1+\epsilon)^{m}}\int_{0}^{1}\frac{b^{2m-1}}{(b+\epsilon)^{m-1}}db\\
   	&+\frac{1}{m(m-1)}\int_{0<a+b<1}\frac{b^{2m-1}}{(a+b+\epsilon)^{m+1}(1-a+\epsilon)^{m-1}}dadb\\
   	:=&\sum_{i=1}^{6}K_{i}^{(2m-1,2)}(\epsilon).
   \end{align*}
   By elementary calculations, we have
   \begin{align*}
   \lim_{\epsilon\rightarrow 0}\frac{K_{1}^{(2m-1,2)}(\epsilon)}{\log(1/\epsilon)}	=\frac{1}{m^{2}},
   \end{align*}
   and
  $$
   \lim_{\epsilon\rightarrow 0}\frac{K_{i}^{(2m-1,2)}(\epsilon)}{\log(1/\epsilon)}=0,
   $$
   for $i=2,\cdots,6$, which implies that
   \begin{align}\label{sec3-eq3.12}
   	\lim_{\epsilon\rightarrow0}\frac{V_{1}^{(2m-1,2)}(\epsilon)}{\log(1/\epsilon)}=\frac{1}{m^{2}}.
   \end{align}
   For the term $V_{2}^{(2m-1,2)}(\epsilon)$, we deduce that
   \begin{align*}
    &V_{2}^{(2m-1,2)}(\epsilon)\\
    &=\frac{1}{m(m-1)}\left[\int_{0}^{1}\frac{b^{2m-1}}{(b+\epsilon)^{2m}}db-\int_{0}^{1}\frac{b^{2m}}{(b+\epsilon)^{2m}}db\right.\\
   	&\quad -\frac{1}{(m-2)}\int_{0}^{1}\frac{b^{2m-1}}{(b+\epsilon)^{2m-1}}db+
   	\frac{1}{(m-2)(1+\epsilon)^{m-2}}\int_{0}^{1}\frac{b^{2m-1}}{(b+\epsilon)^{m+1}}db\\
   	&\quad
   	 -\frac{1}{(1+\epsilon)^{m-1}}\int_{0}^{1}\frac{b^{2m-1}}{(b+\epsilon)^{m+1}}db+\frac{1}{(1+\epsilon)^{m-1}}\int_{0}^{1}\frac{b^{2m}}{(b+\epsilon)^{m+1}}db\\
   	 &\left.\quad +\frac{1}{(m-2)}\int_{0}^{1}\frac{b^{2m-1}}{(b+\epsilon)^{2m-1}}db-\frac{1}{(m-2)(1+\epsilon)^{m-2}}\int_{0}^{1}\frac{b^{2m-1}}{(b+\epsilon)^{m+1}}db\right]\\
   	:&=\sum_{i=1}^{6}\frac{1}{m(m-1)}G_{i}^{(2m-1,2)}(\epsilon).
   \end{align*}
   It is not hard to see that
   \begin{align*}
    \lim_{\epsilon\rightarrow 0}\frac{G_{1}^{(2m-1,2)}(\epsilon)}{\log(1/\epsilon)}=1,
   \end{align*}
   and
   $$
    \lim_{\epsilon\rightarrow0}\frac{G_{i}^{(2m-1,2)}}{\log(1/\epsilon)}=0,
    $$
   for $i=2,\cdots,8$, which leads to
   \begin{align}\label{sec3-eq3.13}
   	\lim_{\epsilon\rightarrow0}\frac{V_{2}^{(2m-1,2)}(\epsilon)}{\log(1/\epsilon)}=\frac{1}{m(m-1)}.
   \end{align}
Together with~\eqref{sec3-eq3.11}-\eqref{sec3-eq3.13}, one can conclude that for $m\geq2$
    \begin{align*}
    	\lim_{\epsilon\rightarrow 0}\left(\log\frac{1}{\epsilon}\right)^{-1}V^{(2m-1,2)}(\epsilon)=2\beta_{2m-1,2}\left[\frac{1}{m^{2}}+\frac{1}{m(m-1)}\right]
    =\frac{2\beta_{2m-1,2}(2m-1)}{m^2(m-1)}.
    \end{align*}

\end{proof}

\begin{proposition}\label{pro3.3}
Let $d\geq2$ and $d+|k|\geq4$. We have for $2m>|k|$,

\begin{align*}
	\lim_{\epsilon\rightarrow 0}\epsilon^{d+|k|-3}V^{(2m-|k|,d)}(\epsilon)=\sigma_{2m-|k|,d}^{2},
	\end{align*}
	where
	\begin{align*}
	\sigma_{2m-|k|,d}^{2}=\beta_{2m-|k|,d}\phi(m,d,|k|),
	\end{align*}
$\beta_{2m-|k|,d}$ is defined in (\ref{beta}) and	
	\begin{align}\label{phi}
	 \phi(m,d,|k|):=\frac{2\Gamma(2m-|k|+1)\Gamma(d+|k|-3)}{\Gamma(2m+d-2)(m+\frac{d}{2}-1)}\left(\frac{1}{m+\frac{d}{2}-1}+\frac{1}{m+\frac{d}{2}-2}\right).
	\end{align}
\end{proposition}
\begin{proof}
For $d\geq2$ and $d+|k|\geq4$, by (\ref{V}), we have
\begin{align*}
V^{(2m-|k|,d)}(\epsilon)&:=2\beta_{2m-|k|,d}\left(
\int_{\mathcal{T}_1}\frac{(t_1-s_2)^{2m-|k|}}{(t_{1}-s_{1}+\epsilon)^{m+\frac{d}{2}}(t_{2}-s_{2}+\epsilon)^{m+\frac{d}{2}}}ds_{1}dt_{1}ds_{2}dt_{2}\right.\\
&\left.\qquad\quad+\int_{\mathcal{T}_2}\frac{(t_2-s_2)^{2m-|k|}}{(t_{1}-s_{1}+\epsilon)^{m+\frac{d}{2}}(t_{2}-s_{2}+\epsilon)^{m+\frac{d}{2}}}ds_{1}dt_{1}ds_{2}dt_{2}\right)
\end{align*}
where $\beta_{2m-|k|,d}$ is defined in (\ref{beta}).
	Using L'H\^{o}pital's rule, we rewrite
	\begin{equation}\label{sec3-eq3.16-25}
	\begin{split}
	&\lim_{\epsilon\rightarrow 0}\epsilon^{d+|k|-3}V^{(2m-|k|,d)}(\epsilon)\\
	&=\lim_{\epsilon\rightarrow 0}2\beta_{2m-|k|,d}\epsilon^{d+|k|-3}\int_{\mathcal{T}_1}\frac{(t_1-s_2)^{2m-|k|}}{(t_{1}-s_{1}+\epsilon)^{m+\frac{d}{2}}(t_{2}-s_{2}+\epsilon)^{m+\frac{d}{2}}}ds_{1}dt_{1}ds_{2}dt_{2}\\
    &\quad+\lim_{\epsilon\rightarrow 0}2\beta_{2m-|k|,d}\epsilon^{d+|k|-3}\int_{\mathcal{T}_2}\frac{(t_2-s_2)^{2m-|k|}}{(t_{1}-s_{1}+\epsilon)^{m+\frac{d}{2}}(t_{2}-s_{2}+\epsilon)^{m+\frac{d}{2}}}ds_{1}dt_{1}ds_{2}dt_{2}\\	 
	&=\lim_{\epsilon\rightarrow 0}2\beta_{2m-|k|,d}\epsilon\int_{\epsilon^{-1}\mathcal{T}_{1}}\frac{(t_1-s_2)^{2m-|k|}}{(t_{1}-s_{1}+1)^{m+\frac{d}{2}}(t_{2}-s_{2}+1)^{m+\frac{d}{2}}}ds_{1}dt_{1}ds_{2}dt_{2}\\
	&\quad
	+\lim_{\epsilon\rightarrow 0}2\beta_{2m-|k|,d}\epsilon\int_{\epsilon^{-1}\mathcal{T}_{2}}\frac{(t_2-s_2)^{2m-|k|}}{(t_{1}-s_{1}+1)^{m+\frac{d}{2}}(t_{2}-s_{2}+1)^{m+\frac{d}{2}}}ds_{1}dt_{1}ds_{2}dt_{2}\\
    :&=\lim_{\epsilon\rightarrow 0}2\beta_{2m-|k|,d}\epsilon\Big[\tilde{V}_{1}^{(2m-|k|,d)}(\epsilon)+\tilde{V}_{2}^{(2m-|k|,d)}(\epsilon)\Big],
	\end{split}	
	\end{equation}
	where $$\epsilon^{-1}\mathcal{T}_{1}=\left\{(s_{1},s_{2},t_{1},t_{2}):\ 0<s_{1}<s_{2}<t_{1}<t_{2}<\frac{1}{\epsilon}\right\},$$
 and
 $$\epsilon^{-1}\mathcal{T}_{2}=\left\{(s_{1},s_{2},t_{1},t_{2}):\ 0<s_{1}<s_{2}<t_{2}<t_{1}<\frac{1}{\epsilon}\right\}.$$
Making the change of variables $s_{2}-s_{1}=a$, $t_{1}-s_{2}=b$ and $t_{2}-t_{1}=c$, thus
	\begin{equation}\label{sec3-eq3.17-25}
	\begin{split}
	 \tilde{V}&_{1}^{(2m-|k|,d)}(\epsilon)\\
  &=\int_{0<s_{1}+a+b+c<\frac{1}{\epsilon}}\frac{b^{2m-|k|}}{(a+b+1)^{m+\frac{d}{2}}(b+c+1)^{m+\frac{d}{2}}}ds_{1}dadbdc\\
	&=\int_{0<a+b+c<\frac{1}{\epsilon}}\left[\frac{1}{\epsilon}-(a+b+c)\right]\frac{b^{2m-|k|}}{(a+b+1)^{m+\frac{d}{2}}(b+c+1)^{m+\frac{d}{2}}}dadbdc.
	\end{split}	
	\end{equation}
	Similarly, one can write
	\begin{equation}\label{sec3-eq3.18-25}
 \begin{split}
	 \tilde{V}&_{2}^{(2m-|k|,d)}(\epsilon)\\
  &=\int_{0<a+b+c<\frac{1}{\epsilon}}\left[\frac{1}{\epsilon}-(a+b+c)\right]\frac{b^{2m-|k|}}{(a+b+c+1)^{m+\frac{d}{2}}(b+1)^{m+\frac{d}{2}}}dadbdc.
   \end{split}
	\end{equation}
	Together with~\eqref{sec3-eq3.16-25}-\eqref{sec3-eq3.18-25}, we obtain
	\begin{align*}
	&\lim_{\epsilon\rightarrow 0}\epsilon^{d+|k|-3}V^{(2m-|k|,d)}(\epsilon)\\
    &=\lim_{\epsilon\rightarrow 0}2\beta_{2m-|k|,d}\left(\int_{0<a+b+c<\frac{1}{\epsilon}}\frac{b^{2m-|k|}}{(a+b+1)^{m+\frac{d}{2}}(b+c+1)^{m+\frac{d}{2}}}dadbdc\right.\\
	&\quad\left.+\int_{0<a+b+c<\frac{1}{\epsilon}}\frac{b^{2m-|k|}}{(a+b+c+1)^{m+\frac{d}{2}}(b+1)^{m+\frac{d}{2}}}dadbdc\right)\\
	&=2\beta_{2m-|k|,d}\left(\int_{\mathbb{R}_{+}^{3}}\frac{b^{2m-|k|}}{(a+b+1)^{m+\frac{d}{2}}(b+c+1)^{m+\frac{d}{2}}}dadbdc\right.\\
	&\quad +\left.\int_{\mathbb{R}_{+}^{3}}\frac{b^{2m-|k|}}{(a+b+c+1)^{m+\frac{d}{2}}(b+1)^{m+\frac{d}{2}}}dadbdc\right).
	\end{align*}
	Then it follows from Lemma \ref{lem5.2-25} that
	 \begin{align*}
	 	&\lim_{\epsilon\rightarrow 0}\epsilon^{d+|k|-3}V^{(2m-|k|,d)}(\epsilon)\\
        &=2\beta_{2m-|k|,d}\left(\frac{\Gamma(2m-|k|+1)\Gamma(d+|k|-3)}{(m+\frac{d}{2}-1)^{2}\Gamma(2m+d-2)}\right.\\
	 	&\ \ \ \ \ \ \quad\quad +\left.\frac{\Gamma(2m-|k|+1)\Gamma(d+|k|-3)}{(m+\frac{d}{2}-1)(m+\frac{d}{2}-2)\Gamma(2m+d-2)}\right)\\
	 	&=\beta_{2m-|k|,d}\phi(m,d,|k|),
	 \end{align*}
     where
     \begin{align*}
	 \phi(m,d,|k|):=\frac{2\Gamma(2m-|k|+1)\Gamma(d+|k|-3)}{\Gamma(2m+d-2)(m+\frac{d}{2}-1)}\left(\frac{1}{m+\frac{d}{2}-1}+\frac{1}{m+\frac{d}{2}-2}\right),
	\end{align*}
and $\Gamma(\cdot)$ is the Gamma function.
\end{proof}

Using the same arguments as in Proposition~\ref{pro3.3}, we can describe the precise asymptotic behavior of the variance of $I_{2m-|k|}(f_{2m-|k|,d,\epsilon})$ for $d=1$ and $|k|\geq3$.
\begin{proposition}\label{pro3.4}
Let $d=1$ and $|k|\geq3$. We have for $2m>|k|$,

\begin{align*}
	\lim_{\epsilon\rightarrow 0}\epsilon^{|k|-2}V^{(2m-|k|,1)}(\epsilon)=\frac{(2m-1)!(k-3)!}{\pi 2^{2m-3}((m-1)!)^{2}}\left(\frac{1}{m-\frac{1}{2}}+\frac{1}{m-\frac{3}{2}}\right).
	\end{align*}
\end{proposition}

\subsection{The proof of Theorem~\ref{th1}}\label{sec3.2}

\begin{proof}[Proof of (i)] For $d=1$ and $|k|=2$, it is sufficient to prove that
\begin{align*}
\lim\limits_{\epsilon\to 0}\left(\log\frac{1}{\epsilon}\right)^{-1}E\Big[\left(\alpha^{(|k|)}_d(\epsilon)-E[\alpha^{(|k|)}_d(\epsilon)]\right)^2\Big]=\infty.
\end{align*}
By Proposition~\ref{pro3.1} and Fatou's lemma, we can write
\begin{align*}
&\lim\limits_{\epsilon\to 0}\left(\log\frac{1}{\epsilon}\right)^{-1}E\Big[\left(\alpha^{(|k|)}_d(\epsilon)-E[\alpha^{(|k|)}_d(\epsilon)]\right)^2\Big]\\
&=\lim\limits_{\epsilon\to 0}\left(\log\frac{1}{\epsilon}\right)^{-1}\sum^\infty_{m=2}E[I_{2m-2}(f_{2m-2,1,\epsilon})^{2}]\\
&\geq \sum^\infty_{m=2}\liminf_{\epsilon\to 0}\left(\log\frac{1}{\epsilon}\right)^{-1}V^{(2m-2,1)}(\epsilon)\\
&=\frac{8}{\pi}\sum^\infty_{m=0}\frac{(2m)!}{2^{2m}(m!)^2},
\end{align*}
and $\sum^\infty_{m=2}\frac{(2m)!}{2^{2m}(m!)^2}=\infty$ by Stirling's formula.

\end{proof}

\begin{proof}[Proof of (ii)] It follows from Proposition~\ref{pro3.2} that
\begin{align*}
&\lim\limits_{\epsilon\to 0}\left(\log\frac{1}{\epsilon}\right)^{-2}E\Big[\left(\alpha^{(|k|)}_d(\epsilon)-E[\alpha^{(|k|)}_d(\epsilon)]\right)^2\Big]\\
&=\lim\limits_{\epsilon\to 0}\left(\log\frac{1}{\epsilon}\right)^{-2}\sum^\infty_{m=1}E[I_{2m-1}(f_{2m-1,2,\epsilon})^{2}]\\
&=\lim\limits_{\epsilon\to 0}\left(\log\frac{1}{\epsilon}\right)^{-2}V^{(1,2)}(\epsilon)\\
&=\frac{1}{4\pi^2},
\end{align*}
for $d=2$ and $|k|=1$. Indeed, the term $\left(\log\frac{1}{\epsilon}\right)^{-2}\sum^\infty_{m=2}E[I_{2m-1}(f_{2m-1,2,\epsilon})^{2}]$ converges to zero as $\epsilon$ tends to zero.
\end{proof}

\begin{proof}[Proof of (iii)]For $d=1$ and $|k|\geq3$, from Proposition~\ref{pro3.4} and Fatou's lemma, we have
\begin{align*}
&\lim\limits_{\epsilon\to 0}\epsilon^{d+|k|-3}E\Big[\left(\alpha^{(|k|)}_d(\epsilon)-E[\alpha^{(|k|)}_d(\epsilon)]\right)^2\Big]\\
&=\lim\limits_{\epsilon\to 0}\epsilon^{|k|-2}\sum_{m>\frac{|k|}{2}}E[I_{2m-|k|}(f_{2m-|k|,1,\epsilon})^{2}]\\
&\geq\sum_{m>\frac{|k|}{2}}\frac{(2m-1)!(k-3)!}{2^{2m-3}\pi((m-1)!)^{2}}\frac{2}{m-\frac{1}{2}}.
\end{align*}
By Stirling's formula, it is obvious that $\sum_{m>\frac{|k|}{2}}\frac{(2m-1)!(k-3)!}{2^{2m-3}\pi((m-1)!)^{2}}\frac{2}{m-\frac{1}{2}}$ is divergent. Similarly, we have for $d=2$ and $|k|\geq2$
 \begin{align*}
&\lim\limits_{\epsilon\to 0}\epsilon^{d+|k|-3}E\Big[\left(\alpha^{(|k|)}_d(\epsilon)-E[\alpha^{(|k|)}_d(\epsilon)]\right)^2\Big]\\
&=\lim\limits_{\epsilon\to 0}\epsilon^{|k|-1}\sum_{m>\frac{|k|}{2}}E[I_{2m-|k|}(f_{2m-|k|,2,\epsilon})^{2}]\\
&\geq\sum_{m>\frac{|k|}{2}}\beta_{2m-|k|,2}\phi(m,2,|k|).
\end{align*}
Using Lemma~\ref{lem5.3-7-26} and (\ref{phi}), we have the following fact that
$$
\beta_{2m-|k|,2}=Cm^{|k|}+O(m^{|k|-1}),
$$
and
$$
\lim_{m\rightarrow\infty}\frac{\beta_{2m-|k|,2}\phi(m,2,|k|)}{\frac{1}{m}}=C.
$$
Hence $$\sum_{m>\frac{|k|}{2}}\beta_{2m-|k|,2}\phi(m,2,|k|)=\infty,$$
and the assertion $(iii)$ in Theorem~\ref{th1} follows.
\end{proof}

\begin{proof}[Proof of (iv)] By Proposition~\ref{pro3.3}, it is enough to show
$$
\sum_{m>\frac{|k|}{2}}\beta_{2m-|k|,d}\phi(m,d,|k|)<\infty.
$$
Together with Lemma~\ref{lem5.3-7-26} and (\ref{phi}), thus
$$
\lim_{m\rightarrow\infty}\frac{\beta_{2m-|k|,d}\phi(m,d,|k|)}{\left(\frac{1}{m}\right)^{\frac{d}{2}}}\le C.
$$
For $d\geq3$, it is obvious that $\sum_{m>\frac{|k|}{2}}\left(\frac{1}{m}\right)^{\frac{d}{2}}$ is convergent. Hence, $$\sum_{m>\frac{|k|}{2}}\beta_{2m-|k|,d}\phi(m,d,|k|)<\infty.$$
\end{proof}

\section{The central limit theorems for $I_{2m-|k|}(f_{2m-|k|,d,\epsilon})$ by renormalization}
 Recall that the Wiener chaos expansion of $\alpha_{d}^{(|k|)}(\epsilon)$ is as follows
 $$
 \alpha_{d}^{(|k|)}(\epsilon)=\sum_{m\geq\frac{|k|}{2}}I_{2m-|k|}(f_{2m-|k|,d,\epsilon}),
 $$
where $d\geq1$ and $|k|\geq1$.

In this section, we present the central limit theorems for $r_{m,d,|k|}(\epsilon)I_{2m-|k|}(f_{2m-|k|,d,\epsilon})$ when renormalized by a multiplicative factor $r_{m,d,|k|}(\epsilon)$ for any $d\geq1$ and any $|k|\geq1$. We find that the renormalized factor $r_{m,d,|k|}(\epsilon)$ is as follows:
 \begin{align*}
 r_{m,d,|k|}(\epsilon)=\left\{
\begin{array}{ll}
\left(\log\frac{1}{\epsilon}\right)^{-\frac{1}{2}},&\quad {\text { $d=1,|k|=2,m\geq2$}},\\
\left(\log\frac{1}{\epsilon}\right)^{-1},&\quad {\text { $d=2,|k|=1,m=1$}},\\
\left(\log\frac{1}{\epsilon}\right)^{-\frac{1}{2}},&\quad {\text { $d=2,|k|=1,m\geq2$}},\\
\epsilon^{\frac{d+|k|-3}{2}},&\quad {\text { $d\geq1,d+|k|\geq4$}},m>\frac{|k|}{2}.\\
\end{array}
\right.
 \end{align*}

Before completing the proof of Theorem~\ref{th1.3-7-26}, we give the following lemma which is helpful in the sequel.
\begin{lemma}{(Fourth moment theorem \cite{Nualart-Peccati-05,Nourdin-Peccati})}\label{lem3-7-26}
Let $n\geq 2$ be a fixed integer. Consider a collection of elements $\{f_{T}, \ T>0\}$ such that $f_{T}\in\mathfrak{H}^{\otimes n}$ for every $T>0$. Assume further that
\begin{align*}
\lim\limits_{T\to\infty}E[|I_{n}(f_{T})|^{2}]=\lim\limits_{T\to\infty}n!\|f_{T}\|_{\mathfrak{H}^{\otimes n}}^{2}=\sigma^{2}.
\end{align*}
Then the following conditions are equivalent:
\begin{itemize}
\item As $T\to\infty$, the sequence $\{I_{n}(f_{T})\}_{T>0}$ converges in distribution to a standard Gaussian random variable $N(0,\sigma^{2})$.
\item $\lim\limits_{T\to\infty}E[|I_{n}(f_{T})|^{4}]=3\sigma^{4}.$
\item For every $p=1,2,\cdots,n-1$, $\lim\limits_{T\to\infty}\|f_{T}\otimes_{p}f_{T}\|_{\mathfrak{H}^{\otimes(2n-2p)}}^{2}=0$.	
\end{itemize}
\end{lemma}

For simplicity, we introduce some notations in the sequel. For $1$-dimensional Brownian motion $\{B_{t},\ t\geq0\}$, we define
\begin{align*}
	K(s_{1},t_{1},s_{2},t_{2}):=E[(B_{t_{1}}-B_{s_{1}})(B_{t_{2}}-B_{s_{2}})]
\end{align*}
and
\begin{align*}
	\mu(x,u_{1},u_{2}):=E[B_{u_{1}}(B_{x+u_{2}}-B_{x})],
\end{align*}
where $x>0$, $u_{1}>0$ and $u_{2}>0$. It is clear that
\begin{align*}
	K(s_{1},t_{1},s_{2},t_{2})=\begin{cases}
		\mu(s_{2}-s_{1},t_{1}-s_{1},t_{2}-s_{2}),&\ \ s_{2}-s_{1}>0,\\
		\mu(s_{1}-s_{2},t_{2}-s_{2},t_{1}-s_{1}),&\ \ s_{1}-s_{2}>0,
	\end{cases}
\end{align*}
and
\begin{align}\label{5-16}
	\mu(x,u_1,u_2)=1_{\{0<u_1-x<u_2\}}(u_1-x)+1_{\{0<u_2<u_1-x\}}u_2.
\end{align}

\subsection{The proof of (i) of Theorem~\ref{th1.3-7-26}}\label{sec4.1}
\begin{proof}[Proof of (i)]
	By Proposition \ref{pro3.1} and Lemma \ref{lem3-7-26}, it is enough to prove for $p=1,2,\dots,2m-3$,
	\begin{align}\label{f}
	\lim_{\epsilon\to 0}\frac{1}{(\log(1/\epsilon))^{2}}\|f_{2m-2,1,\epsilon}\otimes_{p}f_{2m-2,1,\epsilon}\|_{\mathfrak{H}^{\otimes(4m-4-2p)}}^{2}=0.
	\end{align}
	The proof will be done in four steps.

   {\bf Step I.}
   Thus we have from the definition of $f_{2m-2,1,\epsilon}$ (see~\eqref{sec2-2})
	\begin{align*}
	&\frac{1}{\log(1/\epsilon)}(f_{2m-2,1,\epsilon}\otimes_{p}f_{2m-2,1,\epsilon})(r_{1},\dots,r_{2m-2-p},v_{1},\dots,v_{2m-2-p})\\
	 &=\frac{\bar{\beta}_{2m-2,1}}{\log(1/\epsilon)}\int_{\mathcal{T}}(\epsilon+t_{1}-s_{1})^{-m-\frac{1}{2}}(\epsilon+t_{2}-s_{2})^{-m-\frac{1}{2}}K(s_{1},t_{1},s_{2},t_{2})^{p}\\
	&\qquad\qquad\qquad\quad \times\prod_{j=1}^{2m-2-p}1_{(s_{1},t_{1}]}(r_{j})1_{(s_{2},t_{2}]}(v_{j})ds_{1}dt_{1}ds_{2}dt_{2},
	\end{align*}
	where
\begin{align*}
\bar{\beta}_{2m-2,1}:=\frac{((2m)!)^{2}}{2\pi2^{2m}((2m-2)!)^{2}(m!)^{2}},
\end{align*}
and
$$\mathcal{T}=\{(s_{1},t_{1},s_{2},t_{2}):0<s_{1}<t_{1}<1,0<s_{2}<t_{2}<1\}.$$
Then we write
	\begin{align*}
    &\frac{1}{(\log(1/\epsilon))^{2}}\|f_{2m-2,1,\epsilon}\otimes_{p}f_{2m-2,1,\epsilon}\|_{\mathfrak{H}^{\otimes(4m-4-2p)}}^{2}\\
    &=\frac{\bar{\beta}_{2m-2,1}^{2}}{(\log(1/\epsilon))^{2}}\int_{\mathcal{D}}\prod_{i=1}^{4}(\epsilon+t_{i}-s_{i})^{-m-\frac{1}{2}}K(s_{1},t_{1},s_{2},t_{2})^{p}K(s_{3},t_{3},s_{4},t_{4})^{p}\\
	&\qquad\qquad\qquad\ \ \ \ \ \ \times K(s_{1},t_{1},s_{3},t_{3})^{2m-2-p}K(s_{2},t_{2},s_{4},t_{4})^{2m-2-p}d\vec{s}d\vec{t},
	\end{align*}
	where $$\mathcal{D}=\{(s_{1},t_{1},s_{2},t_{2},s_{3},t_{3},s_{4},t_{4}):0<s_{i}<t_{i}<1, i=1,2,3,4\}.$$
and $d\vec{s}d\vec{t}:=ds_{1}dt_{1}ds_{2}dt_{2}ds_{3}dt_{3}ds_{4}dt_{4}$. Let $\{B_{t}^{i,j}$, $t\geq 0\}$, $i=1,2,3,4$, $j=1,2,\dots,2m-2$ be 1-dimensional independent Brownian motions. By the independence of Brownian motion, we have $E[B_{t}^{i,j}B_{s}^{i',j'}]=0$ for $(i,j)\neq(i',j')$. It is clear to see that
	\begin{align*}
	K(s_{1},t_{1},s_{2},t_{2})^{p}=E\left[\prod_{l=1}^{p}(B_{t_{1}}^{1,l}-B_{s_{1}}^{1,l})(B_{t_{2}}^{1,l}-B_{s_{2}}^{1,l})\right],
	\end{align*}
	\begin{align*}
	K(s_{3},t_{3},s_{4},t_{4})^{p}=E\left[\prod_{l=1}^{p}(B_{t_{3}}^{2,l}-B_{s_{3}}^{2,l})(B_{t_{4}}^{2,l}-B_{s_{4}}^{2,l})\right],
	\end{align*}
	\begin{align*}
	K(s_{1},t_{1},s_{3},t_{3})^{2m-2-p}=E\left[\prod_{q=1}^{2m-2-p}(B_{t_{1}}^{3,q}-B_{s_{1}}^{3,q})(B_{t_{3}}^{3,q}-B_{s_{3}}^{3,q})\right],
	\end{align*}
	\begin{align*}
	K(s_{2},t_{2},s_{4},t_{4})^{2m-2-p}=E\left[\prod_{q=1}^{2m-2-p}(B_{t_{2}}^{4,q}-B_{s_{2}}^{4,q})(B_{t_{4}}^{4,q}-B_{s_{4}}^{4,q})\right].
	\end{align*}
	Then, one can rewrite	
	\begin{equation}\label{f-xi}
	\begin{split}
	&\frac{1}{(\log(1/\epsilon))^{2}}\|f_{2m-2,1,\epsilon}\otimes_{p}f_{2m-2,1,\epsilon}\|_{\mathfrak{H}^{\otimes(4m-4-2p)}}^{2}\\
	 &=\frac{\bar{\beta}_{2m-2,1}^{2}}{(\log(1/\epsilon))^{2}}E\Big[\int_{\mathcal{D}}\prod_{i=1}^{4}(\epsilon+t_{i}-s_{i})^{-m-\frac{1}{2}}\prod_{l=1}^{p}\prod_{q=1}^{2m-2-p}(B_{t_{1}}^{1,l}-B_{s_{1}}^{1,l})(B_{t_{1}}^{3,q}-B_{s_{1}}^{3,q})(B_{t_{2}}^{1,l}-B_{s_{2}}^{1,l})\\
	&\ \ \ \ \ \ \ \quad\quad\quad \times(B_{t_{2}}^{4,q}-B_{s_{2}}^{4,q})(B_{t_{3}}^{2,l}-B_{s_{3}}^{2,l})(B_{t_{3}}^{3,q}-B_{s_{3}}^{3,q})(B_{t_{4}}^{2,l}-B_{s_{4}}^{2,l})(B_{t_{4}}^{4,q}-B_{s_{4}}^{4,q})d\vec{s}d\vec{t}\Big]\\
	&:=\bar{\beta}_{2m-2,1}^{2}E\Big[\prod_{(i,j)}\xi_{2m-2,1,\epsilon}^{(i,j)}\Big],
	\end{split}	
	\end{equation}
	where $(i,j)$ belongs to $\{(1,3),(1,4),(2,3),(2,4)\}$ and
	\begin{align*}
	 \xi_{2m-2,1,\epsilon}^{(i,j)}=\frac{1}{\sqrt{\log\frac{1}{\epsilon}}}\int_{0}^{1}\int_{0}^{t}(\epsilon+t-s)^{-m-\frac{1}{2}}\prod_{l=1}^{p}\prod_{q=1}^{2m-2-p}(B_{t}^{i,l}-B_{s}^{i,l})(B_{t}^{j,q}-B_{s}^{j,q})dsdt.
	\end{align*}
By the proof of Proposition \ref{pro3.1}, it is not hard to obtain
	\begin{align*}
	\lim_{\epsilon\to 0}E[(\xi_{2m-2,1,\epsilon}^{(i,j)})^{2}]
	&=\lim_{\epsilon\to 0}\frac{1}{\log(1/\epsilon)}\int_{\mathcal{T}}\prod^2_{i=1}(\epsilon+t_{i}-s_{i})^{-m-\frac{1}{2}}K(s_{1},t_{1},s_{2},t_{2})^{2m-2}ds_{1}dt_{1}ds_{2}dt_{2}\\
	&=\bar{\sigma}_{2m-2,1}^{2},
	\end{align*}
	where $\bar{\sigma}_{2m-2,1}^{2}:=2\left(\frac{1}{(m-\frac{1}{2})^{2}}+\frac{1}{(m-\frac{1}{2})(m-\frac{3}{2})}\right)$. It follows from the independence of Brownian motion that $E[\xi_{2m-2,1,\epsilon}^{(i,j)}\xi_{2m-2,1,\epsilon}^{(i',j')}]=0$, for $(i,j)\neq (i',j')$. So we reduce our problem to show that
\begin{align}\label{xi}
	\xi_{2m-2,1,\epsilon}^{(i,j)}\xrightarrow{Law} \mathcal{N}\left(0,\bar{\sigma}_{2m-2,1}^{2}\right), \quad \text{ as }\  \epsilon\rightarrow 0.
	\end{align}
	Indeed, if (\ref{xi}) holds, thus  $$\left(\xi^{(1,3)}_{2m-2,1,\epsilon},\xi^{(1,4)}_{2m-2,1,\epsilon},\xi^{(2,3)}_{2m-2,1,\epsilon},\xi^{(2,4)}_{2m-2,1,\epsilon}\right)$$
converges to a centered Gaussian vector from Proposition $1$ in \cite{Pecatti}. Together with (\ref{f-xi}), we can get (\ref{f}).
	
	{\bf Step II.} By the definition of $\xi_{2m-2,1,\epsilon}^{(i,j)}$, it is clear that random variables $\left\{\xi_{2m-2,1,\epsilon}^{(1,3)},\xi_{2m-2,1,\epsilon}^{(1,4)},\right.$
$\left.\xi_{2m-2,1,\epsilon}^{(2,3)},\xi_{2m-2,1,\epsilon}^{(2,4)}\right\}$ are identically distributed. Due to the self-similarity of Brownian motion, they have the same distribution as $\eta_{2m-2,1,\epsilon}$, where
	\begin{align*}
	 \eta_{2m-2,1,\epsilon}:=\frac{\sqrt{\epsilon}}{\sqrt{\log\frac{1}{\epsilon}}}\int_{0<s<t<\frac{1}{\epsilon}}(1+t-s)^{-m-\frac{1}{2}}\prod_{i=1}^{2m-2}(B_{t}^{i}-B_{s}^{i})dsdt,
	\end{align*} 	
	and $\{B_{t}=(B_{t}^{1},\cdots,B_{t}^{2m-2}),\ t\geq 0\}$ is a $(2m-2)$-dimensional Brownian motion. It is obvious that
	\begin{align*}
	 \eta_{2m-2,1,\epsilon}&=\frac{\sqrt{\epsilon}}{\sqrt{\log\frac{1}{\epsilon}}}\int_{0<s<t\wedge\frac{1}{\epsilon}}(1+t-s)^{-m-\frac{1}{2}}\prod_{i=1}^{2m-2}(B_{t}^{i}-B_{s}^{i})dsdt\\&\ \ \ \ \ \ -\frac{\sqrt{\epsilon}}{\sqrt{\log\frac{1}{\epsilon}}}\int_{0<s<\frac{1}{\epsilon}<t}(1+t-s)^{-m-\frac{1}{2}}\prod_{i=1}^{2m-2}(B_{t}^{i}-B_{s}^{i})dsdt\\
	:&=\eta_{2m-2,1,\epsilon,1}-\eta_{2m-2,1,\epsilon,2}.
	\end{align*}
	By the elementary calculations, we have
	\begin{align*}
	\|\eta_{2m-2,1,\epsilon,2}\|_{L^{2}}&\leq \frac{\sqrt{\epsilon}}{\sqrt{\log\frac{1}{\epsilon}}}\int_{0<s<\frac{1}{\epsilon}<t}(1+t-s)^{-m-\frac{1}{2}}\left\|\prod_{i=1}^{2m-2}(B_{t}^{i}-B_{s}^{i})\right\|_{L^{2}}dsdt\\
	&=\frac{\sqrt{\epsilon}}{\sqrt{\log\frac{1}{\epsilon}}}\int_{0<s<\frac{1}{\epsilon}<t}(1+t-s)^{-m-\frac{1}{2}}(t-s)^{m-1}dsdt\\
	&\leq\frac{\sqrt{\epsilon}}{\sqrt{\log\frac{1}{\epsilon}}}\int_{0<s<\frac{1}{\epsilon}<t}(t-s)^{-\frac{3}{2}}dsdt
	=\frac{4}{\sqrt{\log\frac{1}{\epsilon}}}\to 0,
	\end{align*}
	as $\epsilon$ tends to zero. Further, it suffices to prove that
	\begin{align*}
	\eta_{2m-2,1,\epsilon,1}\xrightarrow{Law} \mathcal{N}\left(0,\bar{\sigma}_{2m-2,1}^{2}\right), \quad \text{ as }\  \epsilon\rightarrow 0.
	\end{align*}
	
	{\bf Step III}
	By changing the variables $u=t-s$ and $N=\frac{1}{\epsilon}$, we rewrite
	\begin{align}\label{35}
	\eta_{2m-2,1,N,1}=\frac{1}{\sqrt{N\log N}}\int_{0}^{\infty}\int_{0}^{N}(1+u)^{-m-\frac{1}{2}}\prod_{i=1}^{2m-2}(B_{s+u}^{i}-B_{s}^{i})dsdu.
	\end{align}
	Denote
	\begin{align}\label{36}
	\zeta_{2m-2,1,N}:=\frac{1}{\sqrt{N\log N}}\int_{1}^{\infty}\int_{0}^{N}u^{-m-\frac{1}{2}}\prod_{i=1}^{2m-2}(B_{s+u}^{i}-B_{s}^{i})dsdu.
	\end{align}
	By Lemma \ref{lem4-7-26}, we know that the asymptotic behavior of $\eta_{2m-2,1,N,1}$ is the same as $\zeta_{2m-2,1,N}$ when $N$ tends to infinity.
	Hence, we will prove that
	\begin{align*}
	\zeta_{2m-2,1,N}\xrightarrow{Law} \mathcal{N}\left(0,\bar{\sigma}_{2m-2,1}^{2}\right), \quad \text{ as }\  N\rightarrow \infty.
	\end{align*}
	By Lemma \ref{lem3-7-26} and Lemma \ref{lem6}, we just need to prove that
	\begin{align}\label{40}
	\lim_{N\to\infty}E[\zeta_{2m-2,1,N}^{4}]=3\bar{\sigma}_{2m-2,1}^{4}.
	\end{align}
	{\bf Step IV}
	Denote
	\begin{align*}
	\vartheta_{2m-2,1,N}(n):=\frac{1}{\sqrt{2^{n}N\log N}}\int_{2^{n}}^{\infty}u^{-m-\frac{1}{2}}\int_{0}^{N}\prod_{i=1}^{2m-2}(B_{s+u}^{i}-B_{s}^{i})dsdu, \ \ \ \ n\geq 1.
	\end{align*}
	It is clear that $\vartheta_{2m-2,1,N}(0)=\zeta_{2m-2,1,N}$. We have that for $n\geq 1$,
	\begin{equation}\label{43}
    \begin{split}
	\vartheta_{2m-2,1,N,}(n-1)&=\frac{1}{\sqrt{2^{n-1}N\log N}}\int_{2^{n-1}}^{\infty}u^{-m-\frac{1}{2}}\int_{0}^{\frac{N}{2}}\prod_{i=1}^{2m-2}(B_{s+u}^{i}-B_{s}^{i})dsdu\\
	&\ \ \ \ \ \ \  +\frac{1}{\sqrt{2^{n-1}N\log N}}\int_{2^{n-1}}^{\infty}u^{-m-\frac{1}{2}}\int_{\frac{N}{2}}^{N}\prod_{i=1}^{2m-2}(B_{s+u}^{i}-B_{s}^{i})dsdu\\
	:&=\vartheta_{2m-2,1,N}(n-1,1)+\vartheta_{2m-2,1,N}(n-1,2).
    \end{split}
	\end{equation}
	It is clear that $\vartheta_{2m-2,1,N}(n-1,1)$ and $\vartheta_{2m-2,1,N}(n-1,2)$ have the same distribution as $\vartheta_{2m-2,1,N}(n)$. That is
	\begin{align}\label{37}
	E[\vartheta^2_{2m-2,1,N}(n-1,1)]=E[\vartheta^2_{2m-2,1,N}(n-1,2)]=E[\vartheta^2_{2m-2,1,N}(n)],
	\end{align}
	which leads to
	\begin{equation}\label{38}
	\begin{split}
	&E[\vartheta^2_{2m-2,1,N}(n-1)]\\
    &=E[\vartheta^2_{2m-2,1,N}(n-1,1)]+E[\vartheta^2_{2m-2,1,N}(n-1,2)]\\
    &\quad+2E[\vartheta_{2m-2,1,N}(n-1,1)\vartheta_{2m-2,1,N}(n-1,2)]\\
	&=2E[\vartheta^2_{2m-2,1,N}(n)]+2E[\vartheta_{2m-2,1,N}(n-1,1)\vartheta_{2m-2,1,N}(n-1,2)].
	\end{split}	
	\end{equation}
	By the proof of Lemma \ref{lem6}, we write for $n\geq0$
	\begin{align}\label{27}
	\lim_{N\to\infty}E[\vartheta^2_{2m-2,1,N}(n)]=\frac{\bar{\sigma}_{2m-2,1}^{2}}{2^{n}}.
	\end{align}
	It follows from (\ref{37}) and (\ref{27}) that
	\begin{align}\label{39}
	\lim_{N\to\infty}E[\vartheta^2_{2m-2,1,N}(n-1,1)]=\lim_{N\to\infty}E[\vartheta^2_{2m-2,1,N}(n-1,2)]=\frac{\bar{\sigma}_{2m-2,1}^{2}}{2^{n}}.
	\end{align}	
	Combining with (\ref{38}) and (\ref{27}), we deduce that
	\begin{align}\label{28}
	\lim_{N\to\infty}E[\vartheta_{2m-2,1,N}(n-1,1)\vartheta_{2m-2,1,N}(n-1,2)]=0.
	\end{align}	 	
    Because of Lemma \ref{lem1-6-16} and (\ref{39})-(\ref{28}), we have, for $n\geq1$,
    \begin{align*}
    &\lim_{N\to\infty}E[\vartheta^2_{2m-2,1,N}(n-1,1)\vartheta^2_{2m-2,1,N}(n-1,2)]\\
    &=\lim_{N\to\infty}E[\vartheta^2_{2m-2,1,N}(n-1,1)]E[\vartheta^2_{2m-2,1,N}(n-1,2)]\\
    &=\lim_{N\to\infty}\left(E[\vartheta^2_{2m-2,1,N}(n)]\right)^{2},
    \end{align*}
    and
    \begin{align*}
    &\lim_{N\to\infty}E[\vartheta_{2m-2,1,N}(n-1,1)\vartheta^3_{2m-2,1,N}(n-1,2)]\\
    &=\lim_{N\to\infty}E[\vartheta^3_{2m-2,1,N}(n-1,1)\vartheta_{2m-2,1,N}(n-1,2)]=0.
    \end{align*}
Hence,
    \begin{align*}
    &E[\vartheta^4_{2m-2,1,N}(n-1)]\\
    &=2E[\vartheta^4_{2m-2,1,N}(n)]+4E[\vartheta_{2m-2,1,N}(n-1,1)\vartheta^3_{2m-2,1,N}(n-1,2)]\\
    &\quad+4E[\vartheta^3_{2m-2,1,N}(n-1,1)\vartheta^3_{2m-2,1,N}(n-1,2)]\\
    &\quad+6E[\vartheta^2_{2m-2,1,N}(n-1,1)\vartheta^2_{2m-2,1,N}(n-1,2)].
    \end{align*}
    Taking limit of $N$ in both sides, we obtain for $n\geq1$
    \begin{equation}\label{30}
    \begin{split}
    &\lim_{N\to\infty}E[\vartheta^4_{2m-2,1,N}(n-1)]\\
    &=2\lim_{N\to\infty}E[\vartheta^4_{2m-2,1,N}(n)]+6\lim_{N\to\infty}\left(E[\vartheta^2_{2m-2,1,N}(n)]\right)^{2}\\
    &=2\lim_{N\to\infty}E[\vartheta^4_{2m-2,1,N}(n)]+\frac{6\bar{\sigma}_{2m-2,1}^{4}}{2^{2n}}.\\
    \end{split}
    \end{equation}
    With the iteration of (\ref{30}), we can write
    \begin{equation}\label{zeta-4}
    \begin{split}
    &\lim_{N\to\infty}E[\zeta_{2m-2,1,N}^{4}]\\
    &=\lim_{N\rightarrow\infty}E[\vartheta^{4}_{2m-2,1,N}(0)]\\
    &=\lim_{n\to\infty}\left(2^{n}\lim_{N\to\infty}E[\vartheta^4_{2m-2,1,N}(n)]+
    \sum_{k=1}^{n}2^{k-1}\frac{6\bar{\sigma}_{2m-2,1}^{4}}{2^{2k}}\right)\\
    &:=\lim_{n\to\infty}[\Lambda_{1}(n)+\Lambda_{2}(n)].
    \end{split}
    \end{equation}
    By the proof of Lemma~\ref{lem1-6-16} and the elementary inequality $(a+b+c)^{2m-2}\leq 3^{2m-3}(a^{2m-2}+b^{2m-2}+c^{2m-2})$, we deduce that $$E[\vartheta_{2m-2,1,N}^{4}(n)]\leq 3^{2m-2}(E[\vartheta_{2m-2,1,N}^{2}(n)])^{2},$$ which implies that
    \begin{align}\label{32}
    \lim_{n\to\infty}\Lambda_{1}(n)\leq \lim_{n\to\infty}\frac{3^{2m-2}\bar{\sigma}_{2m-2,1}^{4}}{2^{n}}=0.
    \end{align}
    It follows from the elementary calculations that
    \begin{equation}\label{33}
    \begin{split}
    \lim_{n\rightarrow\infty}\Lambda_{2}(n)&=\lim_{n\rightarrow\infty}\sum_{k=1}^{n}2^{k-1}\frac{6\bar{\sigma}_{2m-2,1}^{2}}{2^{2k}}\\
    &=\lim_{n\rightarrow\infty}3\bar{\sigma}_{2m-2,1}^{4}\left(1-\frac{1}{2^{n}}\right)\\
    &=3\bar{\sigma}_{2m-2,1}^{4}.
    \end{split}
    \end{equation}
    Together with (\ref{zeta-4})-(\ref{33}), we obtain~(\ref{40}). Thus the proof is completed.
\end{proof}

\subsection{The proof of (ii) of Theorem~\ref{th1.3-7-26}}\label{sec4.2}
\begin{proof}[Proof of (ii)] It is clear that $(\log\frac{1}{\epsilon})^{-1}I_1(f_{1,2,\epsilon})$ is a centered Gaussian random variable. It follows from Proposition~\ref{pro3.2} that
\begin{equation*}
\left(\log\frac{1}{\epsilon}\right)^{-1}I_{1}(f_{1,2,\epsilon})\xrightarrow{Law} \mathcal{N}\left(0,\frac{1}{4\pi^2}\right),
\end{equation*}
when $\epsilon$ tends to $0$.

Next, we address the case of $(2m-1)$-th chaotic components for $m\geq2$. By the definition of $f_{2m-1,2,\epsilon}$ (see~\eqref{sec2-1}), we have
\begin{align*}
&(f_{2m-1,2,\epsilon}\otimes_{p}f_{2m-1,2,\epsilon})(\mathbf{i}_{2m-1-p,2},\mathbf{j}_{2m-1-p,2},r_{1},\cdots,r_{2m-1-p},v_{1},\cdots,v_{2m-1-p})\\
&=\bar{\beta}_{2m-1-p,2}\int_{\mathcal{T}}(\epsilon+t_{1}-s_{1})^{-m-1}(\epsilon+t_{2}-s_{2})^{-m-1}K(s_{1},t_{1},s_{2},t_{2})^{p}\\
&\qquad\qquad \times\prod_{j=1}^{2m-1-p}1_{(s_{1},t_{1}]}(r_{j})1_{(s_{2},t_{2}]}(v_{j})ds_{1}dt_{1}ds_{2}dt_{2},
\end{align*}
where
\begin{align*}
\bar{\beta}_{2m-1-p,2}=\frac{\sum_{\mathbf{k}_{p,2}}A(\mathbf{i}_{2m-1-p,2},\mathbf{k}_{p,2})A(\mathbf{j}_{2m-1-p,2},\mathbf{k}_{p,2})}{(2\pi)^{2}((2m-|k|)!)^{2}},
\end{align*}
and we write
\begin{equation*}
\begin{split}
&\frac{1}{(\log(1/\epsilon))^{2}}\|f_{2m-1,2,\epsilon}\otimes_{p}f_{2m-1,2,\epsilon}\|_{\mathfrak{H}^{\otimes(4m-2-2p)}}^{2}\\
&=\frac{\tilde{\beta}_{2m-1-p,2}}{(\log(1/\epsilon))^{2}}E\Big[\int_{\mathcal{D}}\prod_{i=1}^{4}(\epsilon+t_{i}-s_{i})^{-m-1}\prod_{l=1}^{p}\prod_{q=1}^{2m-1-p}(B_{t_{1}}^{1,l}-B_{s_{1}}^{1,l})(B_{t_{1}}^{3,q}-B_{s_{1}}^{3,q})\\
&\ \ \ \ \ \ \ \quad\quad\quad \times(B_{t_{2}}^{1,l}-B_{s_{2}}^{1,l})(B_{t_{2}}^{4,q}-B_{s_{2}}^{4,q})(B_{t_{3}}^{2,l}-B_{s_{3}}^{2,l})(B_{t_{3}}^{3,q}-B_{s_{3}}^{3,q})\\
&\ \ \ \ \ \ \ \quad\quad\quad \times(B_{t_{4}}^{2,l}-B_{s_{4}}^{2,l})(B_{t_{4}}^{4,q}-B_{s_{4}}^{4,q})ds_{1}dt_{1}ds_{2}dt_{2}ds_{3}dt_{3}ds_{4}dt_{4}\Big]\\
:&=\tilde{\beta}_{2m-1-p,2}E\Big[\prod_{(i,j)}\xi_{2m-1,2,\epsilon}^{(i,j)}\Big],
\end{split}	
\end{equation*}
where $$\mathcal{D}=\{(s_{1},t_{1},s_{2},t_{2},s_{3},t_{3},s_{4},t_{4}):0<s_{i}<t_{i}<1, i=1,2,3,4\},$$
$(i,j)$ belongs to $\{(1,3),(1,4),(2,3),(2,4)\}$ and
\begin{align*}
\xi_{2m-1,2,\epsilon}^{(i,j)}=\frac{1}{\sqrt{\log\frac{1}{\epsilon}}}\int_{0}^{1}\int_{0}^{t}(\epsilon+t-s)^{-m-1}\prod_{l=1}^{p}\prod_{q=1}^{2m-1-p}(B_{t}^{i,l}-B_{s}^{i,l})(B_{t}^{j,q}-B_{s}^{j,q})dsdt.
\end{align*}
Here the constant $\tilde{\beta}_{2m-1-p,2}$ is defined as follows
\begin{align*}
\tilde{\beta}_{2m-1-p,2}&:=\sum_{\substack{\mathbf{i}_{2m-1-p,2}\\\mathbf{j}_{2m-1-p,2}}}\bar{\beta}_{2m-1-p,2}^{2}\\
&=\sum_{\substack{\mathbf{i}_{2m-1-p,2}\\\mathbf{j}_{2m-1-p,2}}}\frac{(\sum_{\mathbf{k}_{p,2}}A(\mathbf{i}_{2m-1-p,2},\mathbf{k}_{p,2})A(\mathbf{j}_{2m-1-p,2},\mathbf{k}_{p,2}))^{2}}{(2\pi)^{4}((2m-|k|)!)^{4}}.
\end{align*}
By the proof of Proposition \ref{pro3.2}, it is not hard to see that
\begin{align*}
&\lim_{\epsilon\to 0}E[(\xi_{2m-1,2,\epsilon}^{(i,j)})^{2}]\\
&=\lim_{\epsilon\to 0}\frac{1}{\log(1/\epsilon)}\int_{\mathcal{T}}\prod^2_{l=1}(\epsilon+t_{l}-s_{l})^{-m-1}K(s_{1},t_{1},s_{2},t_{2})^{2m-1}ds_{1}dt_{1}ds_{2}dt_{2}\\
&=\bar{\sigma}_{2m-1,2}^{2},
\end{align*}
where $\bar{\sigma}_{2m-1,2}^{2}=2\left(\frac{1}{m^{2}}+\frac{1}{m(m-1)}\right)$. Similar as the proof of (i) of Theorem \ref{th1.3-7-26}, we have
\begin{align*}
\xi_{2m-1,2,\epsilon}^{(i,j)}\xrightarrow{Law} \mathcal{N}\left(0,\bar{\sigma}_{2m-1,2}^{2}\right), \quad \text{ as }\  \epsilon\rightarrow 0,
\end{align*}
which implies that
\begin{align*}
\frac{1}{(\log(1/\epsilon))^{2}}\|f_{2m-1,2,\epsilon}\otimes_{p}f_{2m-1,2,\epsilon}\|_{\mathfrak{H}^{\otimes(4m-2-2p)}}^{2}\rightarrow 0,
\end{align*}
when $\epsilon$ approaches to $0$. Together with Lemma \ref{lem3-7-26}, we obtain
\begin{align*}
\left(\log\frac{1}{\epsilon}\right)^{-\frac{1}{2}}I_{2m-1}(f_{2m-1,2,\epsilon})\xrightarrow{Law} \mathcal{N}\left(0,2\beta_{2m-1,2}\Big[\frac{1}{m^{2}}+\frac{1}{m(m-1)}\Big]\right), \quad \text{ as }\  \epsilon\rightarrow 0,
\end{align*}
where $\beta_{2m-1,2}$ is defined in (\ref{014}).
\end{proof}

\subsection{The proof of (iii) of Theorem~\ref{th1.3-7-26}}\label{sec4.3}

\begin{proof}[Proof of (iii)] 
Following from the definition of $f_{2m-|k|,d,\epsilon}$ (see~\eqref{sec2-1}), we have
\begin{align*}
&(f_{2m-|k|,d,\epsilon}\otimes_{p}f_{2m-|k|,d,\epsilon})(\mathbf{i}_{2m-|k|-p,d},\mathbf{j}_{2m-|k|-p,d},r_{1},\cdots,r_{2m-|k|-p},v_{1},\cdots,v_{2m-|k|-p})\\
&=\bar{\beta}_{2m-|k|-p,d}\int_{\mathcal{T}}\frac{K(s_{1},t_{1},s_{2},t_{2})^{p}\prod_{j=1}^{2m-|k|-p}1_{(s_{1},t_{1}]}(r_{j})1_{(s_{2},t_{2}]}(v_{j})}{(t_{1}-s_{1}+\epsilon)^{m+\frac{d}{2}}(t_{2}-s_{2}+\epsilon)^{m+\frac{d}{2}}}ds_{1}dt_{1}ds_{2}dt_{2},
\end{align*}
where
\begin{align*}
\bar{\beta}_{2m-|k|-p,d}=\frac{\sum_{\mathbf{k}_{p,d}}A(\mathbf{i}_{2m-|k|-p,d},\mathbf{k}_{p,d})A(\mathbf{j}_{2m-|k|-p,d},\mathbf{k}_{p,d})}{(2\pi)^{d}((2m-|k|)!)^{2}}.
\end{align*}
It is easy to see that
\begin{align*}
&\epsilon^{2(d+|k|-3)}\|f_{2m-|k|,d,\epsilon}\otimes_{p}f_{2m-|k|,d,\epsilon}\|_{\mathfrak{H}^{\otimes(4m-2|k|-2p)}}^{2}\\
&=\epsilon^{2(d+|k|-3)}\tilde{\beta}_{2m-|k|-p,d}E\int_{\mathcal{D}}\prod_{i=1}^{4}(\epsilon+t_{i}-s_{i})^{-m-\frac{d}{2}}\prod_{l=1}^{p}\prod_{q=1}^{2m-|k|-p}(B_{t_{1}}^{1,l}-B_{s_{1}}^{1,l})(B_{t_{1}}^{3,q}-B_{s_{1}}^{3,q})\\
&\qquad\qquad\qquad\ \ \ \ \ \ \  \times(B_{t_{2}}^{1,l}-B_{s_{2}}^{1,l})(B_{t_{2}}^{4,q}-B_{s_{2}}^{4,q})(B_{t_{3}}^{2,l}-B_{s_{3}}^{2,l})(B_{t_{3}}^{3,q}-B_{s_{3}}^{3,q})\\
&\qquad\qquad\qquad\ \ \ \ \ \ \  \times(B_{t_{4}}^{2,l}-B_{s_{4}}^{2,l})(B_{t_{4}}^{4,q}-B_{s_{4}}^{4,q})ds_{1}dt_{1}ds_{2}dt_{2}ds_{3}dt_{3}ds_{4}dt_{4}\\
&:=\tilde{\beta}_{2m-|k|-p,d}E\Big[\prod_{(i,j)}\xi_{2m-|k|,d,\epsilon}^{(i,j)}\Big],
\end{align*}
where$$\mathcal{D}=\{(s_{1},t_{1},s_{2},t_{2},s_{3},t_{3},s_{4},t_{4}):0<s_{i}<t_{i}<1, i=1,2,3,4\},$$
$(i,j)$ belongs to $\{(1,3),(1,4),(2,3),(2,4)\}$ and
\begin{align*}
\xi_{2m-|k|,d,\epsilon}^{(i,j)}:=\epsilon^{\frac{d+|k|-3}{2}}\int_{0}^{1}\int_{0}^{t}(\epsilon+t-s)^{-m-\frac{d}{2}}\prod_{l=1}^{p}\prod_{q=1}^{2m-|k|-p}(B_{t}^{i,l}-B_{s}^{i,l})(B_{t}^{j,q}-B_{s}^{j,q})dsdt.
\end{align*}
Here the constant $\tilde{\beta}_{2m-|k|-p,d}$ is defined as follows
\begin{align*}
	\tilde{\beta}_{2m-|k|-p,d}&:=\sum_{\substack{\mathbf{i}_{2m-|k|-p,d}\\\mathbf{j}_{2m-|k|-p,d}}}\bar{\beta}_{2m-|k|-p,d}^{2}\\
	 &=\sum_{\substack{\mathbf{i}_{2m-|k|-p,d}\\\mathbf{j}_{2m-|k|-p,d}}}\frac{(\sum_{\mathbf{k}_{p,d}}A(\mathbf{i}_{2m-|k|-p,d},\mathbf{k}_{p,d})A(\mathbf{j}_{2m-|k|-p,d},\mathbf{k}_{p,d}))^{2}}{(2\pi)^{2d}((2m-|k|)!)^{4}}.
\end{align*}
Proposition \ref{pro3.3} yields
\begin{align*}
	 &\lim_{\epsilon\rightarrow0}E[(\xi_{2m-|k|,d,\epsilon}^{(i,j)})^{2}]\\
&=\lim_{\epsilon\rightarrow0}\epsilon^{|k|+d-3}\int_{\mathcal{T}}\prod^2_{i=1}(\epsilon+t_{i}-s_{i})^ {-m-\frac{d}{2}}K(s_{1},t_{1},s_{2},t_{2})^{2m-|k|}ds_{1}dt_{1}ds_{2}dt_{2}\\
	&=\phi(m,d,|k|),
\end{align*}
where $$\phi(m,d,|k|):=\frac{2\Gamma(2m-|k|+1)\Gamma(d+|k|-3)}{\Gamma(2m+d-2)(m+\frac{d}{2}-1)}\left(\frac{1}{m+\frac{d}{2}-1}+\frac{1}{m+\frac{d}{2}-2}\right).$$
Then it is sufficient to prove that
\begin{align*}
	\xi_{2m-|k|,d,\epsilon}^{(i,j)}\xrightarrow{Law} \mathcal{N}\left(0,\phi(m,d,|k|)\right),
\end{align*}
when $\epsilon$ tends to $0$. The proof will be divided into three steps.

{\bf Step I.}
Let $\{B_{t}=(B_{t}^{1},\cdots,B_{t}^{2m-|k|}),\ t\geq 0\}$ be a $(2m-|k|)$-dimensional Brownian motion. By the self-similarity of Brownian motion, we define
\begin{align*}
	\eta_{2m-|k|,d,\epsilon}:=\epsilon^{\frac{1}{2}}\int_{0<s<t<\frac{1}{\epsilon}}(1+t-s)^{-m-\frac{d}{2}}\prod_{i=1}^{2m-|k|}(B_{t}^{i}-B_{s}^{i})dsdt,
\end{align*}
and its distribution is the same as the one of $\xi_{2m-|k|,d,\epsilon}^{(i,j)}$. Then we rewrite
\begin{align*}
	\eta_{2m-|k|,d,\epsilon}&=\epsilon^{\frac{1}{2}}\int_{0<s<t\wedge\frac{1}{\epsilon}}(1+t-s)^{-m-\frac{d}{2}}\prod_{i=1}^{2m-|k|}(B_{t}^{i}-B_{s}^{i})dsdt\\&\ \ \ \ \ \ -\epsilon^{\frac{1}{2}}\int_{0<s<\frac{1}{\epsilon}<t}(1+t-s)^{-m-\frac{d}{2}}\prod_{i=1}^{2m-|k|}(B_{t}^{i}-B_{s}^{i})dsdt\\
	&:=\eta_{2m-|k|,d,\epsilon,1}-\eta_{2m-|k|,d,\epsilon,2}.
\end{align*}
By the elementary calculations, we have
\begin{align*}
	\|\eta_{2m-|k|,d,\epsilon,2}\|_{L^{2}}&\leq \epsilon^{\frac{1}{2}}\int_{0<s<\frac{1}{\epsilon}<t}(1+t-s)^{-m-\frac{d}{2}}\left\|\prod_{i=1}^{2m-|k|}(B_{t}^{i}-B_{s}^{i})\right\|_{L^{2}}dsdt\\
	&=\epsilon^{\frac{1}{2}}\int_{0<s<\frac{1}{\epsilon}<t}(1+t-s)^{-m-\frac{d}{2}}(t-s)^{m-\frac{|k|}{2}}dsdt\\
	&\leq\epsilon^{\frac{1}{2}}\int_{0<s<\frac{1}{\epsilon}<t}(1+t-s)^{-\frac{d+|k|}{2}}dsdt\\
	&=\epsilon^{\frac{1}{2}}\int_{0}^{\frac{1}{\epsilon}}\frac{1}{(\frac{d+|k|}{2}-1)(1+\frac{1}{\epsilon}-s)^{\frac{d+|k|}{2}-1}}ds.
\end{align*}
For $d+|k|=4$, it is trivial that
\begin{align*}
	\|\eta_{2m-|k|,d,\epsilon,2}\|_{L^{2}}\leq\epsilon^{\frac{1}{2}}\log(1+\frac{1}{\epsilon})\rightarrow 0, \quad \text{ as }\  \epsilon\rightarrow 0.
\end{align*}
For $d+|k|>4$, we deduce that
\begin{align*}
	 \|\eta_{2m-|k|,d,\epsilon,2}\|_{L^{2}}&\leq\frac{\epsilon^{\frac{1}{2}}}{(\frac{d+|k|}{2}-1)(\frac{d+|k|}{2}-2)}\Big[1-\frac{1}{(1+\frac{1}{\epsilon})^{\frac{d+|k|}{2}-2}}\Big]\rightarrow 0, \text{ as }\  \epsilon\rightarrow 0.
\end{align*}
Hence, it is enough to show that
\begin{align*}
	\eta_{2m-|k|,d,\epsilon,1}\xrightarrow{Law} \mathcal{N}(0,\phi(m,d,|k|)), \quad \text{ as } \epsilon\rightarrow 0.
\end{align*}

{\bf Step II.}
Making the change of variable $u=t-s$, $N=\frac{1}{\epsilon}$, we can rewrite
\begin{align*}
	\eta_{2m-|k|,d,N,1}&=\frac{1}{\sqrt{N}}\int_{0}^{\infty}(1+u)^{-m-\frac{d}{2}}\int_{0}^{N}\prod_{i=1}^{2m-|k|}(B_{s+u}^{i}-B_{s}^{i})dsdu\\
	&=\int_{0}^{\infty}(1+u)^{-m-\frac{d}{2}}\gamma_{2m-|k|,N}(u)du,
\end{align*}
where
\begin{align*}
	\gamma_{2m-|k|,N}(u):=\frac{1}{\sqrt{N}}\int_{0}^{N}\prod_{i=1}^{2m-|k|}(B_{s+u}^{i}-B_{s}^{i})ds.
\end{align*}
Following from (\ref{07}), we have
\begin{align*}
&\lim_{N\rightarrow\infty}E[\eta_{2m-|k|,d,N,1}^{2}]\\
&=2\int_{0}^{\infty}\int_{0}^{\infty}(1+u_{1})^{-m-\frac{d}{2}}(1+u_{2})^{-m-\frac{d}{2}}\int_{0}^{\infty}\mu(x,u_{1},u_{2})^{2m-|k|}dxdu_{1}du_{2}\\
&=\phi(m,d,|k|),
\end{align*}
where $\mu(x,u_{1},u_{2})$ is defined in~\eqref{sec4.1}. Now we reduce our problem to show that
\begin{align}\label{011}
\eta_{2m-|k|,d,N,1}\xrightarrow{Law} \mathcal{N}(0,\phi(m,d,|k|)), \quad \text{ as }\ N\rightarrow \infty.
\end{align}


{\bf Step III.}
By Proposition $4$ in \cite{Hu-Nualart-05}, there exists a centered Gaussian process $\gamma_{2m-|k|}(u)$ with covariance function
\begin{align*}
\int_{0}^{\infty}\mu(x,u_{1},u_{2})^{2m-|k|}dx,
\end{align*}
such that for every $u>0$,
\begin{align*}
\gamma_{2m-|k|,N}(u)\xrightarrow{Law}\gamma_{2m-|k|}(u),\quad \text{ as }\  N\rightarrow\infty,
\end{align*}
where $\mu(x,u_{1},u_{2})$ is defined in \eqref{5-16}. For $M>0$, we denote
\begin{align*}
	\eta_{2m-|k|,d,N,1}^{(M)}:=\int_{0}^{M}(1+u)^{-m-\frac{d}{2}}\gamma_{2m-|k|,N}(u)du.
\end{align*}
and
\begin{align*}
\vartheta_{2m-|k|,d}^{(M)}:=\int_{0}^{M}(1+u)^{-m-\frac{d}{2}}\gamma_{2m-|k|}(u)du.
\end{align*}
It is not hard to see that $\vartheta_{2m-|k|,d}^{(M)}$ is centered Gaussian and
\begin{align*}
\vartheta_{2m-|k|,d}^{(M)}\sim \mathcal{N}\left(0,2\int_{0}^{M}\int_{0}^{M}(1+u_{1})^{-m-\frac{d}{2}}(1+u_{2})^{-m-\frac{d}{2}}\int_{0}^{\infty}\mu(x,u_{1},u_{2})^{2m-|k|}dxdu_{1}du_{2}\right).
\end{align*}
Thus we define $\vartheta_{2m-|k|,d}$ as follows
\begin{align*}
\vartheta_{2m-|k|,d}:=\int_{0}^{\infty}(1+u)^{-m-\frac{d}{2}}\gamma_{2m-|k|}(u)du,
\end{align*}
and $\vartheta_{2m-|k|,d}$ is a centered Gaussian random variable with variance $\phi(m,d,|k|)$. Let $\psi$ be a continuous and bounded function. It follows from the elementary calculations that
\begin{align*}
	&\left|E[\psi(\eta_{2m-|k|,d,N,1})]-E[\psi(\vartheta_{2m-|k|,d})]\right|\\
	&\leq \left|E[\psi(\eta_{2m-|k|,d,N,1})-\psi(\eta_{2m-|k|,d,N,1}^{(M)})]\right|+\left|E[\psi(\eta_{2m-|k|,d,N,1}^{(M)})-\psi(\vartheta_{2m-|k|,d}^{(M)})]\right|\\
	&\quad+\left|E[\psi(\vartheta_{2m-|k|,d}^{(M)})-\psi(\vartheta_{2m-|k|,d})]\right|\\
	&:=\sum_{i=1}^{3}Q_{M,N}^{i}.
\end{align*}
For the term $Q_{M,N}^{1}$, by (\ref{gamma-N-finite}), we have
\begin{align*}
&\sup_{N}P\left\{\Big|\eta_{2m-|k|,d,N,1}-\eta^{(M)}_{2m-|k|,d,N,1}\Big|>\epsilon\right\}\\
&\qquad\le \frac{1}{\epsilon}\int_{M}^{\infty}(1+u)^{-m-\frac{d}{2}}\sup_{N}E[\gamma_{2m-|k|,N}(u)]du\rightarrow 0,
\end{align*}
 as $M\rightarrow \infty$. From the continuity of $\psi$, we have, as $M\rightarrow \infty$
\begin{align*}
\sup_{N}P\left\{\Big|\psi(\eta_{2m-|k|,d,N,1})-\psi(\eta^{(M)}_{2m-|k|,d,N,1})\Big|>\epsilon\right\}\rightarrow 0.
\end{align*}
Using the dominated convergence theorem, we get
\begin{align*}
\lim_{M\rightarrow\infty}\sup_{N}Q_{M,N}^{1}\le\lim_{M\rightarrow\infty}\sup_{N}E\left[\left|\psi(\eta_{2m-|k|,d,N,1})-\psi(\eta^{(M)}_{2m-|k|,d,N,1})\right|\right]=0.
\end{align*}
Similarly, by (\ref{gamma-finite}), we have
\begin{align*}
\lim_{M\rightarrow\infty}Q_{M,N}^{3}\le\lim_{M\rightarrow\infty}E[|\psi(\vartheta_{2m-|k|,d}^{(M)})-\psi(\vartheta_{2m-|k|,d})|]=0.
\end{align*}
Finally, we handle the term $Q_{M,N}^{2}$. Using Proposition $1$ in \cite{Pecatti}, we have, for fixed $M>0$,
\begin{align*}
\sum_{i=0}^{2^{n}M}\Big(1+\frac{i}{2^{n}}\Big)^{-m-\frac{d}{2}}\gamma_{2m-|k|,N}\Big(\frac{i}{2^{n}}\Big)\xrightarrow{Law} \sum_{i=0}^{2^{n}M}\Big(1+\frac{i}{2^{n}}\Big)^{-m-\frac{d}{2}}\gamma_{2m-|k|}\Big(\frac{i}{2^{n}}\Big)
\end{align*}
when $N$ tends to $\infty$. It is easy to see that, for any $\epsilon>0$,
\begin{align*}
\lim_{n\rightarrow \infty}\lim_{N\rightarrow \infty}P\left\{\Big|\eta_{2m-|k|,d,N,1}^{(M)}-\frac{1}{2^{n}}\sum_{i=0}^{2^{n}M}\Big(1+\frac{i}{2^{n}}\Big)^{-m-\frac{d}{2}}\gamma_{2m-|k|,N}\Big(\frac{i}{2^{n}}\Big)\Big|>\epsilon\right\}=0
\end{align*}
and
\begin{align*}
\lim_{n\rightarrow \infty}P\left\{\Big|
\vartheta_{2m-|k|,d}^{(M)}-\frac{1}{2^{n}}\sum_{i=0}^{2^{n}M}\Big(1+\frac{i}{2^{n}}\Big)^{-m-\frac{d}{2}}\gamma_{2m-|k|}\Big(\frac{i}{2^{n}}\Big)\Big|>\epsilon\right\}=0.
\end{align*}
Combining with the above estimates and Theorem 1.4.2 in \cite{Billingsle}, we have, for fixed $M>0$,  $$\eta_{2m-|k|,d,N,1}^{(M)}\xrightarrow{Law} \vartheta_{2m-|k|,d}^{(M)}\ \  \text{as} \ \ N\rightarrow \infty,$$
which implies
$$\lim_{N\rightarrow\infty}Q^{2}_{M,N}=\lim_{N\rightarrow\infty}\left|E[\psi(\eta_{2m-|k|,d,N,1}^{(M)})-\psi(\vartheta_{2m-|k|,d}^{(M)})]\right|=0.$$
Hence, we obtain (\ref{011}) and the proof is completed.
\end{proof}	

\subsection{The proof of Theorem~\ref{th-sum-convergence}}\label{sec4.4}

In this subsection, we first introduce the following lemma which is given by Hu and Nualart~\cite{Hu-Nualart-05} and helpful in the process of expounding and proving Theorem~\ref{th-sum-convergence}. In other words, we just need to check the four conditions listed in the Lemma~\ref{lemma-Hu-CLT} in order to finish the proof of Theorem~\ref{th-sum-convergence}.

\begin{lemma}\label{lemma-Hu-CLT}
	Consider a sequence of square integrable and centered random variables $\{F_{i},\ i\ge1\}$ with Wiener chaos expansions
	\begin{align*}
		F_{i}=\sum_{n=1}^{\infty}I_{n}(f_{n,i}).
	\end{align*}
Suppose that:

\quad(i) $\lim_{N\rightarrow\infty}\limsup_{i\rightarrow\infty}\sum_{n=N+1}^{\infty}n!\|f_{n,i}\|_{\mathfrak{H}^{\otimes n}}^{2}=0;$\\

\quad(ii)
For every $n\geq1$, $\lim_{i\rightarrow\infty}n!\|f_{n,i}\|_{\mathfrak{H}^{\otimes n}}^{2}=\sigma_{n}^{2}$;\\

\quad(iii) $\sum_{n=1}^{\infty}\sigma_{n}^{2}=\sigma^{2}<\infty$;\\

\quad(iv)
For all $n\geq2$, $p=1,\cdots,n-1$, $\lim_{i\rightarrow\infty}\|f_{n,i}\otimes_{p}f_{n,i}\|_{\mathfrak{H}^{\otimes2(n-p)}}^{2}=0.$\\
Then $F_{i}$ converges in distribution to the normal law $N(0,\sigma^{2})$ as $i$ tends to infinity.
\end{lemma}
\vspace{1mm}

\begin{proof}[Proof of Theorem~\ref{th-sum-convergence}]
	
(i) It follows from that the assertion (ii) in Theorem~\ref{th1} and the assertion (ii) in Theorem~\ref{th1.3-7-26}.

(ii) By the proof of Proposition \ref{pro3.3}, we have
		\begin{align*}
		(2m-|k|)!\|f_{2m-|k|,d,\epsilon}\|^{2}_{\mathfrak{H}^{\otimes n}}\leq\beta_{2m-|k|,d}\phi(m,d,|k|).
		\end{align*}
		Together with Lemma \ref{lem5.3-7-26}, thus
		\begin{align*}
		\sum_{m>\frac{|k|}{2}}\beta_{2m-|k|,d}\phi(m,d,|k|)\leq \sum_{m>\frac{|k|}{2}}\frac{C}{m^{\frac{d}{2}}}<\infty,
		\end{align*}
where $d\geq3$ and $|k|\geq1$. Then it means that the condition (i) and (iii) of Lemma \ref{lemma-Hu-CLT} hold. Proposition \ref{pro3.3} yields (ii) of Lemma \ref{lemma-Hu-CLT}. The assertion (iv) of Lemma \ref{lemma-Hu-CLT} follows from the proof of (iii) in Theorem~\ref{th1.3-7-26}. Hence, the proof is complete.
\end{proof}

\section{Technical lemmas}	
In this section, we prove some technical results that were used to determine the precise asymptotic behaviors for $\alpha_{d}^{(|k|)}(\epsilon)-E[\alpha_{d}^{(|k|)}(\epsilon)]$ and its chaotic components.

\begin{lemma}\label{lastsec-lem2}
	Let the integer $m\geq 2$ be fixed. Then, for any $\epsilon>0$, we have
	\begin{align*}
	 \frac{\epsilon}{(m-\frac{1}{2})(m-\frac{3}{2})}\int_{0<x+y<\frac{1}{\epsilon}}\frac{y^{2m-2}}{(x+y+1)^{m+\frac{1}{2}}(1+\frac{1}{\epsilon}-x)^{m-\frac{3}{2}}}dxdy\leq \frac{1}{(m-\frac{1}{2})(m-\frac{3}{2})^{2}}.
	\end{align*}
\end{lemma}
\begin{proof}	
For any $\epsilon>0$, we have
\begin{align*}	 &\frac{\epsilon}{(m-\frac{1}{2})(m-\frac{3}{2})}\int_{0<x+y<\frac{1}{\epsilon}}\frac{y^{2m-2}}{(x+y+1)^{m+\frac{1}{2}}(1+\frac{1}{\epsilon}-x)^{m-\frac{3}{2}}}dxdy\\ &=\frac{1}{(m-\frac{1}{2})(m-\frac{3}{2})}\int_{0<a+b<1}\frac{b^{2m-2}}{(a+b+\epsilon)^{m+\frac{1}{2}}(1+\epsilon-a)^{m-\frac{3}{2}}}dadb\\ &=\frac{1}{(m-\frac{1}{2})(m-\frac{3}{2})}\int_{0}^{1}\int_{0}^{1-a}\frac{b^{2m-2}}{(a+b+\epsilon)^{m+\frac{1}{2}}(1+\epsilon-a)^{m-\frac{3}{2}}}dbda\\
&\leq\frac{1}{(m-\frac{1}{2})(m-\frac{3}{2})}\int_{0}^{1}\int_{0}^{1-a}\frac{b^{m-\frac52}}{(1+\epsilon-a)^{m-\frac{3}{2}}}dbda\\
&=\frac{1}{(m-\frac{1}{2})(m-\frac{3}{2})^2}\int_{0}^{1}\frac{(1-a)^{m-\frac32}}{(1+\epsilon-a)^{m-\frac{3}{2}}}da\\
&\leq\frac{1}{(m-\frac{1}{2})(m-\frac{3}{2})^2}.
\end{align*}
\end{proof}
\begin{lemma}\label{lem5.2-25}
For $d\geq 1$ and $d+|k|\geq 4$, we then have
\begin{equation}\label{sec5-eq5.1-25}
\begin{split}
\int_{\mathbb{R}_{+}^{3}}\frac{b^{2m-|k|}}{(a+b+1)^{m+\frac{d}{2}}(b+c+1)^{m+\frac{d}{2}}}dadbdc= \frac{\Gamma(2m-|k|+1)\Gamma(d+|k|-3)}{(m+\frac{d}{2}-1)^{2}\Gamma(2m+d-2)},
\end{split}
\end{equation}
    and
    \begin{equation}\label{sec5-eq5.2-25}
    \begin{split}
   \int_{\mathbb{R}_{+}^{3}}\frac{b^{2m-|k|}}{(a+b+c+1)^{m+\frac{d}{2}}(b+1)^{m+\frac{d}{2}}}dadbdc=\frac{\Gamma(2m-|k|+1)\Gamma(d+|k|-3)}{(m+\frac{d}{2}-1)(m+\frac{d}{2}-2)\Gamma(2m+d-2)},
    \end{split}
    \end{equation}
    where $2m>|k|$ and $\Gamma(\cdot)$ is the Gamma function.
\end{lemma}
\begin{proof}
	The elementary calculations lead to
	\begin{align*}
	&\int_{\mathbb{R}_{+}^{3}}\frac{b^{2m-|k|}}{(a+b+1)^{m+\frac{d}{2}}(b+c+1)^{m+\frac{d}{2}}}dadbdc\\
	&=\int_{\mathbb{R}_{+}^{2}}\frac{b^{2m-|k|}}{(m+\frac{d}{2}-1)(a+b+1)^{m+\frac{d}{2}}(b+1)^{m+\frac{d}{2}-1}}dadb\\
	&=\int_{\mathbb{R}_{+}}\frac{b^{2m-|k|}}{(m+\frac{d}{2}-1)^{2}(b+1)^{2m+d-2}}db.
	\end{align*}
By \cite{Gradstein} (p.322, 3.241.4$^{11}$ with $p=q=\nu=1$, replace $\mu$ with $2m-|k|+1$ and $n$ with $2m+d-3$), we have
    \begin{align*}\label{gamma}
    	\int_{\mathbb{R}_{+}}\frac{b^{2m-|k|}}{(b+1)^{2m+d-2}}db=\frac{\Gamma(2m-|k|+1)\Gamma(d+|k|-3)}{\Gamma(2m+d-2)}.
    \end{align*}
    The above two identities yield~\eqref{sec5-eq5.1-25}. With the same arguments, we can get \eqref{sec5-eq5.2-25}.
\end{proof}

\begin{lemma}\label{lem5.3-7-26}
For $d\geq 1$ and $|k|\geq 1$, there exist positive constants $c$ and $C$ such that
		\begin{align*}
			c\leq \lim_{m\rightarrow\infty}\frac{\beta_{2m-|k|,d}}{m^{|k|+\frac{d}{2}-1}}\leq C,
		\end{align*}
where $m\in\mathbb{N}$, $2m>|k|$ and
        \begin{align*}
		\beta_{2m-|k|,d}:=\sum_{\substack{m_{1}+\cdots+m_{d}=m\\m_{j}\ge\frac{k_{j}}{2},\ j=1,\cdots,d}}\frac{((2m_{1})!\cdots(2m_{d})!)^{2}}{(2\pi)^{d}2^{2m}(2m_{1}-k_{1})!\cdots(2m_{d}-k_{d})!(m_{1}!)^{2}\cdots(m_{d}!)^{2}}.
		\end{align*}			
\end{lemma}
\begin{proof}
	By the definition of $\beta_{2m-|k|,d}$, we have
    \begin{align*}
    	&\beta_{2m-|k|,d}
    	=\sum_{\substack{m_{1}+\cdots+m_{d}=m\\m_{j}\ge\frac{k_{j}}{2},\ j=1,\cdots,d}}F_1(m_1)F_2(m_2)\cdots F_d(m_d)\frac{(2m_{1})!\cdots(2m_{d})!}{(m_{1}!)^{2}\cdots(m_{d}!)^{2}2^{2m}},
    \end{align*}
    where
    $$
    F_j(m_j)=\left\{
\begin{array}{ll}
(2m_{j})(2m_j-1)\cdots(2m_{j}-k_{j}+1),&\quad {\text { $k_j\geq1$}},\\
1,&\quad {\text { $k_j=0$}},
\end{array}
\right.
    $$
and it is trivial that $\beta_{2m-|k|,d}$ is a polynomial function of a single independent variable $m$. It is not hard to see that
 $$
\lim_{m\rightarrow\infty}\frac{(2m)!\sqrt{m}}{(m!)^22^{2m}}=\frac{1}{\sqrt{\pi}},
$$
and there exists positive constant $\hat{c}$ such that
 $$
\frac{(2m)!\sqrt{m}}{(m!)^22^{2m}}\geq\hat{c}>0,
 $$
for any integer $m\geq1$. Now our problem is reduced to
$$
\lim_{m\rightarrow\infty}\frac{1}{m^{|k|+\frac{d}{2}-1}}\sum_{\substack{m_{1}+\cdots+m_{d}=m\\m_{j}\ge\frac{k_{j}}{2},\ j=1,\cdots,d}}
\frac{F_1(m_1)F_2(m_2)\cdots F_d(m_d)}{\sqrt{m_1}\sqrt{m_2}\cdots\sqrt{m_d}}=C.
$$
When $k_j>1$ for $j=1,2,\ldots,d$, we can write
\begin{equation}\label{sec5-5.3-7-26}
\begin{split}
    	&\frac{1}{m^{|k|+\frac{d}{2}-1}}\sum_{\substack{m_{1}+\cdots+m_{d}=m\\m_{j}\ge\frac{k_{j}}{2},\ j=1,\cdots,d}}\frac{F_1(m_1)F_2(m_2)\cdots F_d(m_d)}{\sqrt{m_1}\sqrt{m_2}\cdots\sqrt{m_d}}\\
    	&=\frac{C}{m^{|k|+\frac{d}{2}-1}}\sum_{\substack{m_{1}+\cdots+m_{d}=m\\m_{j}\ge\frac{k_{j}}{2},\ j=1,\cdots,d}}\frac{m_{1}^{k_{1}}\cdots m_{d}^{k_{d}}}{\sqrt{m_{1}\cdots m_{d}}}\\
    	&\quad+\frac{C}{m^{|k|+\frac{d}{2}-1}}\sum_{\substack{m_{1}+\cdots+m_{d}=m\\m_{j}\ge\frac{k_{j}}{2},\ j=1,\cdots,d}}\frac{(\sum_{i_{1}=1}^{k_{1}-1}(-1)^{i_{1}}C_{i_{1}}m_{1}^{k_{1}-i_{1}})\cdots(\sum_{i_{d}=1}^{k_{d}-1}(-1)^{i_{d}}C_{i_{d}}m_{d}^{k_{d}-i_{d}})}{\sqrt{m_{1}\cdots m_{d}}}\\
    	&=\frac{C}{m^{|k|+\frac{d}{2}-1}}\sum_{\substack{m_{1}+\cdots+m_{d}=m\\m_{j}\ge\frac{k_{j}}{2},\ j=1,\cdots,d}}m_{1}^{k_{1}-\frac{1}{2}}\cdots m_{d}^{k_{d}-\frac{1}{2}}\\
    	&\quad+\frac{C}{m^{|k|+\frac{d}{2}-1}}\sum_{\substack{m_{1}+\cdots+m_{d}=m\\m_{j}\ge\frac{k_{j}}{2},\ j=1,\cdots,d}}\sum_{i_{1}=1}^{k_{1}-1}\cdots\sum_{i_{d}=1}^{k_{d}-1}(-1)^{i_{1}+\cdots+i_{d}}C_{i_{1}}\cdots C_{i_{d}}m_{1}^{k_{1}-i_{1}-\frac{1}{2}}\cdots m_{d}^{k_{d}-i_{d}-\frac{1}{2}}.
   \end{split}
    \end{equation}
When $k_{j}=0\text{ or }1$ for $j\in\{1,2,\ldots,d\}$, we replace $\sum_{i_{j}=1}^{k_{j}-1}(-1)^{i_{j}}C_{i_{j}}m_{j}^{k_{j}-i_{j}}$ by $1$ in~\eqref{sec5-5.3-7-26}, and the equality~\eqref{sec5-5.3-7-26} still holds. The first term of the last equality is asymptotically equivalent with the sequence of integrals
    \begin{align*}
    	&\frac{C}{m^{|k|+\frac{d}{2}-1}}\int_{\{x_{j}\geq \frac{k_{j}}{2},\ j=1,\cdots,d-1,\ x_{1}+\cdots+x_{d-1}\leq m\}}x_{1}^{k_{1}-\frac{1}{2}}\cdots x_{d-1}^{k_{d-1}-\frac{1}{2}}\\
     &\qquad\qquad\qquad\qquad\quad\times(m-x_{1}-\cdots-x_{d-1})^{k_{d}-\frac{1}{2}}dx_{1}\cdots dx_{d-1}\\
    	&=\frac{Cm^{|k|-\frac{d}{2}+d-1}}{m^{|k|+\frac{d}{2}-1}}\int_{\{x_{j}\geq \frac{k_{j}}{2m},\ j=1,\cdots,d-1,\ x_{1}+\cdots+x_{d-1}\leq 1\}}x_{1}^{k_{1}-\frac{1}{2}}\cdots x_{d-1}^{k_{d-1}-\frac{1}{2}}\\
      &\qquad\qquad\qquad\qquad\quad\times(1-x_{1}-\cdots-x_{d-1})^{k_{d}-\frac{1}{2}}dx_{1}\cdots dx_{d-1}\\
    	&\rightarrow C,
    \end{align*}
    as $m\rightarrow\infty$ due to the fact that the last integral is finite. Similarly, the second term can be written
    \begin{align*}
    	&\frac{C}{m^{|k|+\frac{d}{2}-1}}\int_{\{x_{j}\geq\frac{k_{j}}{2},\ j=1,\cdots,d-1,\ x_{1}+\cdots+x_{d-1}\leq m\}}\sum_{i_{1}=1}^{k_{1}-1}\cdots\sum_{i_{d}=1}^{k_{d}-1}(-1)^{i_{1}+\cdots+i_{d}}C_{i_{1}}\cdots C_{i_{d}}x_{1}^{k_{1}-i_{1}-\frac{1}{2}}\\
     &\qquad\qquad\qquad\ \ \ \ \ \ \ \ \ \ \times\cdots x_{d-1}^{k_{d-1}-i_{d-1}-\frac{1}{2}}
    	(m-x_{1}-\cdots-x_{d-1})^{k_{d}-i_{d}-\frac{1}{2}}dx_{1}\cdots dx_{d-1}\\    	    	
    	&=\sum_{i_{1}=1}^{k_{1}-1}\cdots \sum_{i_{d}=1}^{k_{d}-1}\frac{(-1)^{i_{1}+\cdots+i_{d}}}{m^{i_{1}+\cdots+i_{d}}}\int_{\{x_{j}\geq \frac{k_{j}}{2m},\ j=1,\cdots,d-1,\ x_{1}+\cdots+x_{d-1}\leq 1\}}x_{1}^{k_{1}-i_{1}-\frac{1}{2}}\cdots x_{d-1}^{k_{d-1}-i_{d-1}-\frac{1}{2}}\\
    	&\qquad\qquad\qquad\ \ \ \ \ \ \ \ \ \ \times(1-x_{1}-\cdots-x_{d-1})^{k_{d}-i_{d}-\frac{1}{2}}dx_{1}\cdots dx_{d-1}\\
     &\rightarrow0,
    \end{align*}
    as $m\rightarrow\infty$. Summing up the above arguments, there exist positive constants $c$ and $C$ such that
    \begin{align*}
    	c\leq \lim_{m\rightarrow\infty}\frac{\beta_{2m-|k|,d}}{m^{|k|+\frac{d}{2}-1}}\leq C.
    \end{align*}
    Thus we complete the proof.
\end{proof}

\begin{lemma}\label{lem4-7-26}
	Let the integer $m\geq 2$ be fixed. $\eta_{2m-2,1,N,1}$ and $\zeta_{2m-2,1,N}$ are respectively defined by
	\begin{align*}
	\eta_{2m-2,1,N,1}:=\frac{1}{\sqrt{N\log N}}\int_{0}^{\infty}\int_{0}^{N}(1+u)^{-m-\frac{1}{2}}\prod_{i=1}^{2m-2}(B_{s+u}^{i}-B_{s}^{i})dsdu,
	\end{align*}
	and
	\begin{align}\label{zeta}
	\zeta_{2m-2,1,N}:=\frac{1}{\sqrt{N\log N}}\int_{1}^{\infty}\int_{0}^{N}u^{-m-\frac{1}{2}}\prod_{i=1}^{2m-2}(B_{s+u}^{i}-B_{s}^{i})dsdu.
	\end{align}
	Then we have
	\begin{align*}
	E[|\eta_{2m-2,1,N,1}-\zeta_{2m-2,1,N}|]\to 0, \ \ \text{as}\ \ N\to\infty.
	\end{align*}
\end{lemma}
\begin{proof}
	By the definition of $\eta_{2m-2,1,N,1}$ and $\zeta_{2m-2,1,N}$, we deduce that
	\begin{align*}
	&E[|\eta_{2m-2,1,N,1}-\zeta_{2m-2,1,N}|]\\
    &\leq \frac{1}{\sqrt{N\log N}}E\left[\int_{0}^{1}(1+u)^{-m-\frac{1}{2}} \Big|\int_{0}^{N}\prod_{i=1}^{2m-2}(B_{s+u}^{i}-B_{s}^{i})ds\Big|du\right]\\
	&\quad +\frac{1}{\sqrt{N\log N}}E\left[\int_{1}^{\infty}\Big(\frac{1}{u^{m+\frac{1}{2}}}-\frac{1}{(1+u)^{m+\frac{1}{2}}}\Big)\Big|\int_{0}^{N}\prod_{i=1}^{2m-2}(B_{s+u}^{i}-B_{s}^{i})ds\Big|du\right]\\
	&:=Q_{1}(N)+Q_{2}(N).
	\end{align*}
	For the first term $Q_{1}(N)$, making the change of variable $x=s_{2}-s_{1}$, we have
	\begin{equation}\label{22}
	\begin{split}
	Q_{1}(N)&\leq\frac{1}{\sqrt{N\log N}}\int_{0}^{1}\frac{1}{(1+u)^{m+\frac{1}{2}}}\left(E\Big[(\int_{0}^{N}\prod_{i=1}^{2m-2}(B_{s+u}^{i}-B_{s}^{i})ds)^{2}\Big]\right)^{\frac{1}{2}}du\\
	&\leq\frac{\sqrt{2}}{\sqrt{N\log N}}\int_{0}^{1}\frac{1}{(1+u)^{m+\frac{1}{2}}}\left(\int_{0}^{N}\int_{0}^{s_{2}}K(s_{1},s_{1}+u,s_{2},s_{2}+u)^{2m-2}ds_{1}ds_{2}\right)^{\frac{1}{2}}du\\
	&=\frac{\sqrt{2}}{\sqrt{\log N}}\int_{0}^{1}\frac{1}{(1+u)^{m+\frac{1}{2}}}\left(\int_{0}^{u}(u-x)^{2m-2}\Big(1-\frac{x}{N}\Big)dx\right)^{\frac{1}{2}}du\\
	&\leq\frac{\sqrt{2}}{\sqrt{(2m-1)\log N}}\int_{0}^{1}\frac{u^{m-\frac{1}{2}}}{(1+u)^{m+\frac{1}{2}}}du\leq\frac{\sqrt{2}\log 2}{\sqrt{(2m-1)\log N}}. 		
	\end{split}
	\end{equation}
	Similarly, we handle the term $Q_{2}(N)$,
	\begin{equation}\label{21}
	\begin{split}
	Q_{2}(N)&\leq\frac{\sqrt{2}}{\sqrt{(2m-1)\log N}}\int_{1}^{\infty}\left(\frac{1}{u^{m+\frac{1}{2}}}-\frac{1}{(1+u)^{m+\frac{1}{2}}}\right)u^{m-\frac{1}{2}}du\\
	&=\frac{\sqrt{2}}{\sqrt{(2m-1)\log N}}\int_{1}^{\infty}\frac{1}{u}\left(1-(\frac{u}{1+u})^{m+\frac{1}{2}}\right)du\\
	&\leq\frac{\sqrt{2}}{\sqrt{(2m-1)\log N}}\int_{1}^{\infty}\frac{m+\frac{1}{2}}{u(1+u)}du\le\frac{\sqrt{2}(m+\frac{1}{2})}{\sqrt{(2m-1)\log N}}.		
	\end{split}
	\end{equation}
	Thus the lemma follows from~(\ref{22}) and (\ref{21}).
\end{proof}
\begin{lemma}\label{lem6}
	Let $m\in\mathbb{N}$, $m\geq2$ be fixed, and $\zeta_{2m-2,1,N}$ is defined in (\ref{zeta}). Then we have
	\begin{align*}
	\lim_{N\to\infty}E[\zeta_{2m-2,1,N}^{2}]=\bar{\sigma}_{2m-2,1}^{2},
	\end{align*}
where $\bar{\sigma}_{2m-2,1}^{2}=2\Big(\frac{1}{(m-\frac{1}{2})^{2}}+\frac{1}{(m-\frac{1}{2})(m-\frac{3}{2})}\Big)$.
\end{lemma}
\begin{proof}
	By the definition of $\zeta_{2m-2,1,N}$, changing the variable $x=s_{2}-s_{1}$, we deduce that
\begin{align*}
&E[\zeta_{2m-2,1,N}^{2}]\\
&=\frac{2}{N\log N}\int_{1}^{\infty}\int_{1}^{\infty}(u_{1}u_{2})^{-m-\frac{1}{2}}\int_{0}^{N}\int_{0}^{s_{2}}K(s_{1},s_{1}+u_{1},s_{2},s_{2}+u_{2})^{2m-2}ds_{1}du_{1}ds_{2}du_{2}\\
&=\frac{2}{\log N}\int_{1}^{\infty}\int_{1}^{\infty}(u_{1}u_{2})^{-m-\frac{1}{2}}\int_{0}^{N}\mu(x,u_{1},u_{2})^{2m-2}\Big(1-\frac{x}{N}\Big)dxdu_{1}du_{2}.
\end{align*}
By L'H\^{o}pital's rule, we have
\begin{align*}
&\lim_{N\to\infty}E[\zeta_{2m-2,1,N}^{2}]\\
&=\lim_{N\to\infty}\frac{2}{\log N}\int_{1}^{\infty}\int_{1}^{\infty}(u_{1}u_{2})^{-m-\frac{1}{2}}\int_{0}^{N}\mu(x,u_{1},u_{2})^{2m-2}dxdu_{1}du_{2}\\
&=\lim_{N\to\infty}2N\int_{1}^{\infty}\int_{1}^{\infty}(u_{1}u_{2})^{-m-\frac{1}{2}}\mu(N,u_{1},u_{2})^{2m-2}du_{1}du_{2}\\
&=2\int_{0}^{\infty}\int_{0}^{\infty}(u_{1}u_{2})^{-m-\frac{1}{2}}\mu(1,u_{1},u_{2})^{2m-2}du_{1}du_{2}\\
&=2\int_{1}^{\infty}\int_{u_{1}-1}^{\infty}\frac{(u_{1}-1)^{2m-2}}{u_{1}^{m+\frac{1}{2}}u_{2}^{m+\frac{1}{2}}}du_{2}du_{1}+2\int_{0}^{\infty}\int_{1+u_{2}}^{\infty}\frac{u_{2}^{2m-2}}{u_{1}^{m+\frac{1}{2}}u_{2}^{m+\frac{1}{2}}}du_{1}du_{2}\\
&=\frac{2}{m-\frac{1}{2}}\int_{1}^{\infty}\frac{(u_{1}-1)^{m-\frac{3}{2}}}{u_{1}^{m+\frac{1}{2}}}du_{1}+\frac{2}{m-\frac{1}{2}}\int_{0}^{\infty}\frac{u_{2}^{m-\frac{5}{2}}}{(1+u_{2})^{m-\frac{1}{2}}}du_{2}\\
&=\frac{2}{m-\frac{1}{2}}\left[\frac{1}{m-\frac{1}{2}}+\frac{1}{m-\frac{3}{2}}\right],
\end{align*}
where we have used the fact that $$\int\frac{(x-1)^{m-\frac{3}{2}}}{x^{m+\frac{1}{2}}}dx=\frac{1}{m-\frac{1}{2}}\Big(\frac{x-1}{x}\Big)^{m-\frac{1}{2}}+C,$$ and $$\int\frac{(x-1)^{m-\frac{5}{2}}}{x^{m-\frac{1}{2}}}dx=\frac{1}{m-\frac{3}{2}}\Big(\frac{x-1}{x}\Big)^{m-\frac{3}{2}}+C.$$
Thus the lemma follows. 	
\end{proof}

\begin{lemma}\label{lem1-6-16}
	Let $m\in\mathbb{N}$, $m\geq2$ be fixed. $\vartheta_{2m-2,1,N}(n-1,1)$ and $\vartheta_{2m-2,1,N}(n-1,2)$ are defined by
	\begin{align*}
	\vartheta_{2m-2,1,N}(n-1,1)=\frac{1}{\sqrt{2^{n-1}N\log N}}\int_{2^{n-1}}^{\infty}u^{-m-\frac{1}{2}}\int_{0}^{\frac{N}{2}}\prod_{i=1}^{2m-2}(B_{s+u}^{i}-B_{s}^{i})dsdu,
	\end{align*}
	and
	\begin{align*}
	\vartheta_{2m-2,1,N}(n-1,2)=\frac{1}{\sqrt{2^{n-1}N\log N}}\int_{2^{n-1}}^{\infty}u^{-m-\frac{1}{2}}\int_{\frac{N}{2}}^{N}\prod_{i=1}^{2m-2}(B_{s+u}^{i}-B_{s}^{i})dsdu.
	\end{align*}
	Then we have
    \begin{equation}\label{34}
    \begin{split}
    &\lim_{N\to\infty}E[\vartheta_{2m-2,1,N}^{2}(n-1,1)\vartheta_{2m-2,1,N}^{2}(n-1,2)]\\
    &=\lim_{N\to\infty}E[\vartheta_{2m-2,1,N}^{2}(n-1,1)]E[\vartheta_{2m-2,1,N}^{2}(n-1,2)],
    \end{split}
    \end{equation}
	\begin{align}\label{41}
	\lim_{N\to\infty}E[\vartheta_{2m-2,1,N}(n-1,1)\vartheta_{2m-2,1,N}^{3}(n-1,2)]=0,
	\end{align}
	and
	\begin{align}\label{42}
	\lim_{N\to\infty}E[\vartheta_{2m-2,1,N}^{3}(n-1,1)\vartheta_{2m-2,1,N}(n-1,2)]=0.
	\end{align}
\end{lemma}
\begin{proof}
	By the definition of $\vartheta_{2m-2,1,N}(n-1,1)$ and $\vartheta_{2m-2,1,N}(n-1,2)$, we have
	\begin{align*}
	&E[\vartheta_{2m-2,1,N}^{2}(n-1,1)\vartheta_{2m-2,1,N}^{2}(n-1,2)]\\
	&=\frac{1}{2^{2(n-1)}N^{2}(\log N)^{2}}E\left[\int_{[2^{n-1},\infty)^{4}\times [0,\frac{N}{2})^{2}\times [\frac{N}{2},N)^{2} }\prod_{j=1}^{4}u_{j}^{-m-\frac{1}{2}}\right.\\
&\left.\qquad\qquad\qquad\qquad\qquad\qquad\qquad\qquad\times\prod_{i=1}^{2m-2}(B_{s_{j}+u_{j}}^{i}-B_{s_{j}}^{i})ds_{j}du_{j}\right]\\
	&=\frac{1}{2^{2(n-1)}N^{2}(\log N)^{2}}\int_{[2^{n-1},\infty)^{4}\times [0,\frac{N}{2})^{2}\times [\frac{N}{2},N)^{2} }\left(E\left[\prod_{l=1}^{4}(B_{s_{l}+u_{l}}-B_{s_{l}})\right]\right)^{2m-2}\\
&\qquad\qquad\qquad\qquad\qquad\qquad\qquad\qquad\times\prod_{j=1}^{4}u_{j}^{-m-\frac{1}{2}}ds_{j}du_{j},
	\end{align*}
	where $s_{1},\ s_{2}\ \in [0,\frac{N}{2})$ and $s_{3},\ s_{4}\ \in [\frac{N}{2},N)$. Using Isserlis theorem, we deduce that
	\begin{align*}
	 &E\left[\prod_{l=1}^{4}(B_{s_{l}+u_{l}}-B_{s_{l}})\right]\\
  &=E[(B_{s_{1}+u_{1}}-B_{s_{1}})(B_{s_{2}+u_{2}}-B_{s_{2}})]E[(B_{s_{3}+u_{3}}-B_{s_{3}})(B_{s_{4}+u_{4}}-B_{s_{4}})]\\
	&\quad+E[(B_{s_{1}+u_{1}}-B_{s_{1}})(B_{s_{3}+u_{3}}-B_{s_{3}})]E[(B_{s_{2}+u_{2}}-B_{s_{2}})(B_{s_{4}+u_{4}}-B_{s_{4}})]\\
	&\quad+E[(B_{s_{1}+u_{1}}-B_{s_{1}})(B_{s_{4}+u_{4}}-B_{s_{4}})]E[(B_{s_{2}+u_{2}}-B_{s_{2}})(B_{s_{3}+u_{3}}-B_{s_{3}})]\\
	:&=\Lambda_{1}+\Lambda_{2}+\Lambda_{3}.
	\end{align*}
	Following from (\ref{28}), we have
	\begin{align*}
	&\frac{1}{2^{2(n-1)}N^{2}(\log N)^{2}}\int_{[2^{n-1},\infty)^{4}\times [0,\frac{N}{2})^{2}\times [\frac{N}{2},N)^{2} }\Lambda_{2}^{2m-2}\prod_{j=1}^{4}u_{j}^{-m-\frac{1}{2}}ds_{j}du_{j}\\
	&=\frac{1}{2^{2(n-1)}N^{2}(\log N)^{2}}\int_{[2^{n-1},\infty)^{4}\times [0,\frac{N}{2})^{2}\times [\frac{N}{2},N)^{2} }\Lambda_{3}^{2m-2}\prod_{j=1}^{4}u_{j}^{-m-\frac{1}{2}}ds_{j}du_{j}\\
	&=\Big(E[\vartheta_{2m-2,1,N}(n-1,1)\vartheta_{2m-2,1,N}(n-1,2)]\Big)^{2} \rightarrow 0, \ \  \text{as}\ \  N\to \infty.
	\end{align*}
	Moreover, by H\"{o}lder's inequality, for any nonnegative integers $p$, $q$ and $l$ such that $p+q+l=2m-2$, we have
	\begin{align*}
	&\frac{1}{2^{2(n-1)}N^{2}(\log N)^{2}}\int_{[2^{n-1},\infty)^{4}\times [0,\frac{N}{2})^{2}\times [\frac{N}{2},N)^{2} }\Lambda_{1}^{p}\Lambda_{2}^{q}\Lambda_{3}^{l}\prod_{j=1}^{4}u_{j}^{-m-\frac{1}{2}}ds_{j}du_{j}\\
	&\leq\frac{1}{2^{2(n-1)}N^{2}(\log N)^{2}}\left(\int_{[2^{n-1},\infty)^{4}\times [0,\frac{N}{2})^{2}\times [\frac{N}{2},N)^{2} }\Lambda_{1}^{2m-2}\prod_{j=1}^{4}u_{j}^{-m-\frac{1}{2}}ds_{j}du_{j}\right)^{\frac{p}{2m-2}}\\
    &\qquad\times\left(\int_{[2^{n-1},\infty)^{4}\times [0,\frac{N}{2})^{2}\times [\frac{N}{2},N)^{2} }\Lambda_{2}^{2m-2}\prod_{j=1}^{4}u_{j}^{-m-\frac{1}{2}}ds_{j}du_{j}\right)^{\frac{q}{2m-2}}\\
&\qquad\times\left(\int_{[2^{n-1},\infty)^{4}\times [0,\frac{N}{2})^{2}\times [\frac{N}{2},N)^{2} }\Lambda_{3}^{2m-2}\prod_{j=1}^{4}u_{j}^{-m-\frac{1}{2}}ds_{j}du_{j}\right)^{\frac{l}{2m-2}}\rightarrow 0, \ \ \text{as}\ \ N\to\infty.
	\end{align*}
	Hence,
	\begin{align*}
	&\lim_{N\to\infty}E[\vartheta_{2m-2,1,N}^{2}(n-1,1)\vartheta_{2m-2,1,N}^{2}(n-1,2)]\\
	&=\lim_{N\to\infty}\frac{1}{2^{2(n-1)}N^{2}(\log N)^{2}}\int_{[2^{n-1},\infty)^{4}\times [0,\frac{N}{2})^{2}\times [\frac{N}{2},N)^{2}}\Lambda_{1}^{2m-2}\prod_{j=1}^{4}u_{j}^{-m-\frac{1}{2}}ds_{j}du_{j}\\
	&=\lim_{N\to\infty}E[\vartheta_{2m-2,1,N}^{2}(n-1,1)]E[\vartheta_{2m-2,1,N}^{2}(n-1,2)].
	\end{align*}
	This completes the proof of (\ref{34}). By the same arguments, we obtain (\ref{41}) and (\ref{42}).
\end{proof}	

\begin{lemma}\label{le5}
	Let $d\geq 1$, $|k|\geq 1$, $d+|k|\geq 4$ be integers. $\gamma_{2m-|k|,N}(u)$ be defined by
	\begin{align}\label{04}
	\gamma_{2m-|k|,N}(u):=\frac{1}{\sqrt{N}}\int_{0}^{N}\prod_{i=1}^{2m-|k|}(B_{s+u}^{i}-B_{s}^{i})ds,
	\end{align}
	and $\gamma_{2m-|k|}(u)$ be a
	centered Gaussian process with covariance function $$2\int_{0}^{\infty}\mu(x,u_{1},u_{2})^{2m-|k|}dx,$$ where $\mu(x,u_{1},u_{2})$ is defined in (\ref{5-16}). Then we have
	\begin{align}\label{gamma-N-finite}
		\int_{0}^{\infty}(1+u)^{-m-\frac{d}{2}}\sup_{N}E[|\gamma_{2m-|k|,N}(u)|]du<\infty,
	\end{align}
	and
	\begin{align}\label{gamma-finite}
	\int_{0}^{\infty}(1+u)^{-m-\frac{d}{2}}E[|\gamma_{2m-|k|}(u)|]du<\infty,
	\end{align}
 where $m\in\mathbb{N}$ and $m>\frac{|k|}{2}$.
\end{lemma}
\begin{proof}
	By the definition of $\gamma_{2m-|k|,N}(u)$, the elementary calculations lead to
	\begin{equation}\label{07}
	\begin{split}
	E[\gamma_{2m-|k|,N}(u)^{2}]&=\frac{1}{N}\int_{0}^{N}\int_{0}^{N}K(s_{1},s_{1}+u,s_{2},s_{2}+u)^{2m-|k|}ds_{1}ds_{2}\\
	&=\frac{2}{N}\int_{0}^{N}\int_{0}^{s_{2}}\mu(s_{2}-s_{1},u,u)^{2m-|k|}ds_{1}ds_{2}\\
	&=2\int_{0}^{N}\mu(x,u,u)^{2m-|k|}\Big(1-\frac{x}{N}\Big)dx,
	\end{split}		
	\end{equation}
	where we take the change of variable $x=s_{2}-s_{1}$. Taking limit of $N$ in both sides, thus
	\begin{equation}\label{08}
 \begin{split}
	 \lim\limits_{N\rightarrow\infty}E[\gamma_{2m-|k|,N}(u)^{2}]
  &=2\int_{0}^{\infty}\mu(x,u,u)^{2m-|k|}dx\\
 &=\frac{2u^{2m-|k|+1}}{2m-|k|+1}=E[\gamma_{2m-|k|}(u)^{2}].
 \end{split}
	\end{equation}
	Using H\"{o}lder's inequality, we have
	\begin{align*}
	&\int_{0}^{\infty}(1+u)^{-m-\frac{d}{2}}\sup_{N}E[|\gamma_{2m-|k|,N}(u)|]du\\
 &\qquad\qquad\leq \int_{0}^{\infty}(1+u)^{-m-\frac{d}{2}}\left(\sup_{N}E[|\gamma_{2m-|k|,N}(u)|^{2}]\right)^{\frac{1}{2}}du.
	\end{align*}	
	It follows from (\ref{07}) and (\ref{08}) that
	\begin{align*}
		&\int_{0}^{\infty}(1+u)^{-m-\frac{d}{2}}\left(\sup_{N}E[|\gamma_{2m-|k|,N}(u)|^{2}]\right)^{\frac{1}{2}}du\\
  &\quad\leq \frac{\sqrt{2}}{\sqrt{2m-|k|+1}}\int_{0}^{\infty}(1+u)^{-m-\frac{d}{2}}u^{m-\frac{|k|-1}{2}}du\\
		&\quad\leq \frac{\sqrt{2}}{\sqrt{2m-|k|+1}}\int_{0}^{\infty}\frac{1}{(1+u)^{\frac{d+|k|-1}{2}}}du\\
		&\quad=\frac{2\sqrt{2}}{\sqrt{2m-|k|+1}(d+|k|-3)}<\infty.
	\end{align*}
	Similarly, we have (\ref{gamma-finite}).	
\end{proof}

\end{document}